\documentclass[12pt]{article}
\usepackage{amsfonts}
\usepackage{amsmath}
\usepackage{amssymb}
\usepackage{amsthm}
\usepackage[english]{babel}
\usepackage[margin=0.9in]{geometry}
\usepackage{graphicx}
\usepackage[colorlinks]{hyperref}
\usepackage[utf8]{inputenc}
\usepackage{mathabx}
\usepackage{mathrsfs}
\usepackage{mathtools}
\usepackage{newtxtext}
\usepackage{titlesec}
\usepackage{verbatim}
\usepackage{wasysym}
\usepackage{xcolor}
\numberwithin{equation}{section}

\newtheorem{definition}[equation]{Definition}
\newtheorem{proposition}[equation]{Proposition}
\newtheorem{theorem}[equation]{Theorem}
\newtheorem{corollary}[equation]{Corollary}
\newtheorem{lemma}[equation]{Lemma}
\newtheorem{remark}[equation]{Remark}

\title{Semi-Continuity of the Morse Index for Ricci Shrinkers}
\author{Louis Yudowitz\footnote{KTH Royal Institute of Technology, Department of Mathematics, 
Lindstedtsv\"agen 25, 114 28 Stockholm, Sweden\\ \indent \indent \emph{E-mail address:} \texttt{yudowitz@kth.se}\\ \indent \indent ORCID ID: 0000-0003-4131-2396}}
\date{\today}

\begin{document}
\maketitle

\begin{abstract}
  We prove lower and upper semi-continuity of the Morse index for sequences of gradient Ricci shrinkers which bubble tree converge in the sense of past work by the author and Buzano. Our proofs rely on adapting recent arguments of Workman which were used to study certain sequences of CMC hypersurfaces and were in turn adapted from work of Da Lio--Gianocca--Rivi\`ere. Moreover, we are able to refine Workman's methods by using techniques related to polynomially weighted Sobolev spaces. This all also requires us to extend the analysis to handle when the shrinkers we study are non-compact, which we can do due to the availability of a suitable notion of finite weighted volume. Finally, we identify a technical condition which ensures the Morse index of an asymptotically conical shrinker is bounded below by the f-index of its asymptotic cone.
\end{abstract}



\tableofcontents

\section{Introduction}\label{introduction}
  In this article we study how the linear stability of sequences of gradient Ricci shrinkers can change when conical singularities form in the limit. These shrinkers are Riemannian manifolds $\left(M,g\right)$ such that, for some \emph{potential function} $f: M \rightarrow \mathbb{R}$, the following identity is satisfied:

  \begin{equation}\label{grs_eqn}
    \mathrm{Ric}\left(g\right) + \nabla^2 f = \frac{g}{2}.
  \end{equation}

  By now, it is known that gradient Ricci shrinkers model many finite time singularities of Ricci flow, a geometric heat equation which evolves a time dependent family of Riemannian metrics according to $\partial_t g\left(t\right) = -2\mathrm{Ric}\left(g\left(t\right)\right)$. Additionally, each shrinker induces a self similarly shrinking Ricci flow with $g\left(t\right) = -t \varphi^\ast_t g\left(-1\right)$, where $t \in \left(-\infty,0\right)$ and $\varphi_t$ the family of diffeomorphisms generated by $\frac{\nabla f}{-t}$ with $\varphi_{-1} = \mathrm{id}_M$. Some easy examples of shrinkers are positive Einstein manifolds (e.g. the round sphere $S^n$) by taking $f$ to be a constant function. Other examples of shrinkers are the round cylinder $S^{n-1} \times \mathbb{R}$, or more generally $M \times \mathbb{R}^k$ with the product metric, $k \in \mathbb{N}$, and $M$ a positive Einstein manifold.

  When studying finite time Ricci flow singularities, one considers ``blow-up limits'' to zoom into the singular region and understand its geometry at small scales. This is done by parabolically rescaling the flow to get a sequence of manifolds $\left(M, g_i\left(t\right),p_i\right)$, where $p_i$ is a sequence of points around which the curvature becomes unbounded as $t \nearrow T$, $g_i\left(t\right) := s^{-2}_i g\left(T + s^2_i t\right)$, and $s_i \rightarrow 0$ is a sequence of scaling factors dictated by the curvature blow-up. If the finite time singularity is ``Type I'', then any blow-up limit will subconverge in the smooth pointed Cheeger--Gromov sense to a non-flat gradient Ricci shrinker (\cite{emt}). On the other hand, Bamler has recently shown that, without the Type I assumption, blow-up limits (taken in a suitable weak sense) are either asymptotic to or are globally a gradient shrinker and have a codimension $4$ singular set in a certain parabolic sense (\cite{bamler_compactness,bamler_nash,bamler_structure}). If one specializes to dimension $n=4$, then the singular set consists of isolated conical singularities of orbifold type. A precise definition of orbifold points is given in Section \ref{bubble_tree_prelims}.

  If one instead focuses on just the space of Ricci shrinkers, orbifold compactness results were known before Bamler's work. In the compact case, these include results due to Cao--Sesum (\cite{CS07_compactness}), and Weber (\cite{We11_compactness}), while in the complete case we note the work by Li--Li--Wang (\cite{li_li_wang_grs}) and the following theorem by Haslhofer--M\"uller from \cite{HM1,HM2}, which is the most pertinent for this paper:

  \begin{theorem}\label{hm_compactness_thm}
    \sloppy Let $n \geq 4$ and let $\left(M_i,g_i,f_i\right)$ be a sequence of $n$-dimensional gradient Ricci shrinkers with entropy uniformly bounded below $\mu\left(g_i\right)\geq \underline{\mu}>-\infty$ which are normalized such that $\left(4\pi\right)^{-\frac{n}{2}}\int_M e^{-f} dV_g = 1$. If $n \geq 5$, then assume in addition a positive constant $E\left(r\right)$ exists such that the following uniform local energy bound holds for every $i$ and each $r > 0$:

    \begin{equation}\label{energybounds}
      \int_{B_{g_i}\left(p_i,r\right)} \left|\mathrm{Rm}_{g_i}\right|^{\frac{n}{2}}_{g_i} dV_{g_i} \leq E\left(r\right) < \infty.
    \end{equation}

    Then, with $p_i$ a minimum of the potential function $f_i$, after passing to a subsequence $\left(M_i,g_i,f_i,p_i\right)$ converges to an orbifold Ricci shrinker $\left(M_\infty,g_\infty,f_\infty,p_\infty\right)$ in the pointed orbifold Cheeger--Gromov sense. 
  \end{theorem}

  Here $\mu\left(g_i\right)$ denotes Perelman's entropy functional, which we define in Section \ref{shrinker_prelims}. The pointed orbifold Cheeger--Gromov convergence mentioned above is defined in Section \ref{bubble_tree_prelims}.

  In light of Theorem \ref{hm_compactness_thm}, a natural question is how precisely orbifold singularities form and how to deal with any degenerations they cause. This was investigated in \cite{by_bubbling} by the author and Buzano. There, it was shown that orbifold singularities form due to energy concentration, which causes a loss of energy (hence possibly also topology by Chern--Gauss--Bonnet) in the limit and that this lost information can be recovered by constructing a ``bubble tree''. All of this is defined and discussed in more depth in Section \ref{bubble_tree_prelims}.

  We now change topic to the stability of singularity models, which is still a wide open question. If a singularity model is unstable, then it could in principle be perturbed away in favor of a ``generic'' Ricci flow. This would pave the way for geometric and topological applications. Stability results for certain types of Einstein manifolds and Ricci solitons have been proved in a variety of places. A non-exhaustive list is: \cite{cao_zhu_lin_stable,der_ozuch_stability,deruelle_soliton_stability,hall_murphy_stability, HM_stability,kroncke_yudowitz_stability,kroncke_einstein_stability,kroncke_soliton_stability}. While we do not prove such results in this paper, we hope that our results and methods will be of help in the future. In particular, we study how linear (in)stability of shrinkers changes under orbifold convergence. By linear stability of a Ricci shrinker we mean the following:

  \begin{definition}\label{defn_lin_stable}
    Let $\left(M,g,f\right)$ be a complete connected gradient Ricci shrinker. Then the shrinker is \emph{linearly stable} if the operator

    \begin{equation*}
      L_f u := \Delta_f u + 2\mathrm{Rm}_g \ast u
    \end{equation*}

    has no negative eigenmodes other than $\mathrm{Ric}_g$ (at least in a suitable weak sense we introduce in Section \ref{grs_functional}). If, other than $\mathrm{Ric}_g$, $L_f$ has only positive eigenvalues then we say the shrinker is \emph{strictly linearly stable}. Any shrinker which is not linearly stable is \emph{linearly unstable}.
  \end{definition}

  In Definition \ref{defn_lin_stable} the drift Laplacian $\Delta_f$ is defined as

  \begin{equation*}
        \Delta_f u := \Delta u - \left<\nabla f, \nabla u\right>
    \end{equation*}

    with $\Delta u := \mathrm{div}_g \nabla_g u$, while the action of $\mathrm{Rm}_g$ on a symmetric $2$-tensor $u$ is defined as

    \begin{equation*}
        \left(\mathrm{Rm}_g \ast u\right)_{ij} := g^{kp}g^{\ell q} \mathrm{Rm}_{ik\ell j}u_{pq}.
    \end{equation*}

  One might worry that the spectrum of $L_f$ could behave quite badly if the associated shrinker is non-compact. However, we will see in Section \ref{grs_functional} that the spectrum of $L_f$ will always be discrete, as long as we work in the weighted Sobolev space $W^{1,2}_f$. This space is defined with respect to a weighted volume measure that arises naturally when studying Ricci shrinkers.

  The notion of stability in Definition \ref{defn_lin_stable} is motivated by looking at the linearization of \eqref{grs_eqn} after modifying it by a natural gauge term (see, for instance, Proposition $2.1$ in \cite{deruelle_soliton_stability}). The assumption about the Ricci curvature is due to $\mathrm{Ric}_g$ always being a negative eigenmode of $L_f$ on any shrinker. Definition \ref{defn_lin_stable} is admittedly a bit unsatsfying, since one would ideally want to involve the linear operator which appears when taking the second variation of the entropy $\mu\left(g\right)$, the critical points of which are precisely gradient Ricci shrinkers. Moreover, having a more systematic way of disregarding the Ricci curvature as an eigenmode would also be desirable. In the compact setting, this can all be done as shown by Cao--Zhu in \cite{cao_zhu_lin_stable}. There, the authors prove that the stability operator $N_f$ appearing in the second variation of $\mu\left(g\right)$ reduces to $L_f$ acting on a space which consists of elements that are orthogonal in some sense to $\mathrm{Ric}_g$ and lie in $\mathrm{ker}\left(\mathrm{div}_{g,f}\right)$, where $\mathrm{div}_{g,f} u := e^f\mathrm{div}_g\left(e^{-f}u\right)$. Also, one can show $N_f \mathrm{Ric}_g = 0$, so in this sense the Ricci curvature is really a neutral eigenmode. As far as the author knows, it is a difficult open problem to show this in the non-compact setting. Specifically, the infimum in the definition of the entropy might not be achieved. In \cite{zhang_entropy}, a condition on the asymptotic geometry is identified which, when satisfied, guarantees the infimum in the definintion of the entropy is achieved.  However, it is not clear how to verify this condition in general and it is violated by the round cylinder $S^{n-1} \times \mathbb{R}$. This means the infimum might not be achieved even in a small neighborhood of very special metrics.

  As mentioned earlier, the work by the author and Buzano in \cite{by_bubbling} shows if a sequence of shrinkers converges to an orbifold, then lost topology and energy can be recovered after constructing the bubble trees. However, there is still the question of how the formation of orbifold singularities affects the linear stability of shrinkers. Addressing this is the content of our main theorem:

  \begin{theorem}\label{semi_continuity_thm}
    For $n \geq 4$, assume $\left(M_i, g_i, f_i\right)$ is a sequence of $n$-dimensional complete connected gradient Ricci shrinkers with $\mu\left(g_i\right) \geq \underline{\mu} > -\infty$, $\left(4\pi\right)^{-\frac{n}{2}} \int_{M_i} e^{-f_i} dV_{g_i} = 1$, $R_{g_i} + \sup_{M_i \backslash K_i} \left|\mathrm{Rm}_{g_i}\right| \leq R_0$ for some compact set $K_i \subset M_i$, and, if $n \geq 5$, satisfying \eqref{energybounds}. Assume also that the sequence bubble tree converges to an orbifold Ricci shrinker $\left(M_\infty,g_\infty,f_\infty\right)$ with singular set $\mathcal{Q}$ and there is a $T \in \mathbb{N}$ such that $\sum_{q \in \mathcal{Q}} N_q = T$, where $N_q$ is the number of bubbles arising from considering the point-scale sequences associated to $q \in \mathcal{Q}$. Then

    \begin{align}
      \mathrm{Ind}_{f_\infty}\left(M_\infty\right) + \sum_{q \in \mathcal{Q}} \sum^{N_q}_{k=1} \mathrm{Ind}\left(V^k\right) &\leq \liminf \limits_{i \rightarrow \infty} \mathrm{Ind}_{f_i}\left(M_i\right), \label{lower}\\
      \limsup \limits_{i \rightarrow \infty} \left(\mathrm{Ind}_{f_i}\left(M_i\right) + \mathrm{Null}_{f_i}\left(M_i\right)\right) \leq \mathrm{Ind}_{f_\infty}\left(M_\infty\right) &+ \mathrm{Null}_{f_\infty}\left(M_\infty\right) + \sum_{q \in \mathcal{Q}} \sum^{N_q}_{k=1} \mathrm{Ind}\left(V^k\right) + \mathrm{Null}\left(V^k\right). \label{upper}
    \end{align}

    Here the index of the bubbles is defined as $\mathrm{Ind}\left(V^k\right) := \lim \limits_{R \rightarrow \infty} \mathrm{Ind}\left(B_{h^k}\left(q^k,R\right)\right)$ for some point $q^k \in V^k$ and the nullity is defined analogously. 
  \end{theorem}

  The objects and terminology in the statement of Theorem \ref{semi_continuity_thm}, in particular the various indexes, nullities, and bubble tree convergence and associated terminology, will be gradually introduced throughout Section \ref{preliminaries} and Section \ref{functional}. Moving forward, we will use $\left(\mathcal{A}\right)$ to denote the following set of assumptions on a sequence of $n$-dimensional complete connected gradient Ricci shrinkers $\left(M_i,g_i,f_i\right)$:

  \begin{itemize}
    \item $\mu\left(g_i\right) \geq \underline{\mu} > -\infty$,
    \item $R_{g_i} + \sup_{M_i \backslash K_i} \left|\mathrm{Rm}_{g_i}\right| \leq R_0$ for some compact set $K_i \subset M_i$,
    \item If $n \geq 5$, then \eqref{energybounds} is satisfied,
    \item There is a $T \in \mathbb{N}$ such that $\sum_{q \in \mathcal{Q}} N_q = T$, where $N_q$ is the number of bubbles arising from considering the point-scale sequences associated to $q \in \mathcal{Q}$.
  \end{itemize}

  Theorem \ref{semi_continuity_thm} is most striking in dimension $n = 4$ and the assumptions also simplify in this dimension. First, the theorem handles the only type of singularities that can be present in blow-up limits of $4$-dimensional closed Ricci flows by Bamler's work in \cite{bamler_structure}. Also, \cite{HM2} shows \eqref{energybounds} is always satisfied in dimension $4$. The asymptotic curvature bound $\sup_{M_i \backslash K_i} \left|\mathrm{Rm}\left(g_i\right)\right| < \infty$, which we need to ensure good functional and spectral analytic properties of $L_{f_i}$, is implied by the scalar curvature bound in dimension $n = 4$ due to work of Munteanu--Wang (Theorem $1.2$ in \cite{mw_conical}). It is conjectured that every ($4$-dimensional) shrinker has bounded scalar curvature, but this is so far only known in the K\"ahler surface case (\cite{li_wang_4d_kahler}) and when $n \leq 3$. It is also worth noting that asymptotically conical shrinkers, in any dimension, always satisfy the asymptotic curvature bound (Proposition $2.1(3)$ in \cite{kw_conical_rigidity}). We assume $R_{g_i} \leq R_0$ and $\sum_{q \in \mathcal{Q}} N_q = T < \infty$ so that the sums involving the index and nullity of all the bubbles are finite. Without this, \eqref{upper} would trivially hold and \eqref{lower} might not hold. On the other hand, these assumptions can be relaxed if one is only interested in local behavior. We discuss this in more depth at the end of Section \ref{bubble_tree_prelims}. Finally, we remark that recent work by Bertellotti--Buzano (Theorem $1.2$ and Remark $4.2$ in \cite{bert_buz}) shows that the assumptions of Theorem \ref{semi_continuity_thm} imply that the number of ends of the limiting orbifold and each shrinker in the sequence is uniformly bounded above by a constant depending only on $n, R_0, \underline{\mu}$. 

  If we assume the $M_i$ are closed, then so is $M_\infty$ by Theorem $1.4$ in \cite{mw_conical}, which proves that the uniform entropy bound implies a uniform diameter bound. This, along with Theorem \ref{semi_continuity_thm}, the lower semi-continuity of the energy, and Corollary $1.5$ from \cite{by_bubbling} yield the following corollary:

  \begin{corollary}\label{cmpct_corollary}
    For $n \geq 4$, let $\mathcal{M}^n\left(\underline{\mu},E,I\right)$ be the collection of closed $n$-dimensional complete connected gradient Ricci shrinkers $\left(M,g,f\right)$ with entropy uniformly bounded below $\mu\left(g\right) \geq \underline{\mu} > -\infty$, $\left(4\pi\right)^{-\frac{n}{2}} \int_M e^{-f} dV_g = 1$, energy uniformly bounded above 

    \begin{equation*}
      \int_M \left|\mathrm{Rm}_g\right|^{\frac{n}{2}} dV_g \leq E < \infty,
    \end{equation*}

    and index uniformly bounded above $\mathrm{Ind}_f\left(M\right) \leq I < \infty$. Then $\mathcal{M}^n\left(\underline{\mu},E,I\right)$ is compact in the orbifold Cheeger--Gromov sense and contains a finite number of diffeomorphism types, the number of which is bounded above by a positive constant depending only on $n, \underline{\mu}, E, I$.
  \end{corollary}

  As far as the author knows, Theorem \ref{semi_continuity_thm} and Corollary \ref{cmpct_corollary} are new even for positive Einstein manifolds, which are special cases of Ricci shrinkers. With a little more work one can incorporate the Betti numbers of $M$ into Corollary \ref{cmpct_corollary}, but we omit this to emphasize the index bound.

  Finally, we can bound the Morse index of an asymptotically conical Ricci shrinker below by the f-index of its asymptotic cone under a technical condition. By a cone we mean a manifold $\mathcal{C}\left(\Sigma\right) := \left[0,\infty\right) \times \Sigma$ endowed with the metric $g_{\mathcal{C}} := dr^2 + r^2 g_\Sigma$, where the cone link $\left(\Sigma, g_\Sigma\right)$ is a closed $\left(n-1\right)$ dimensional Riemannian manifold. Next, define $E_R := \left(R,\infty\right) \times \Sigma$ for $R \geq 0$. Further let $\mathcal{E}$ be an end of a manifold $M$ and $\rho_\lambda: E_0 \rightarrow E_0$ be a dialtion map given by $\rho_\lambda\left(r,x\right) := \left(\lambda r, x\right)$ for every $x \in \Sigma$ and any $\lambda > 0$. 

  \begin{definition}\label{asymp_con_defn}
    We say an end $\mathcal{E}$ of a manifold $\left(M,g\right)$ is conical in the $C^k$-sense if, along $\mathcal{E}$ and for some $R > 0$, there is a diffeomorphism $\Phi: E_R \rightarrow \mathcal{E}$ such that $\lambda^{-2}\rho^\ast_\lambda \Phi^\ast g \rightarrow g_{\mathcal{C}}$ in the $C^k_{\mathrm{loc}}$-sense as $\lambda \rightarrow \infty$. We say a manifold is \textit{asymptotically conical} if each end is conical.
  \end{definition}

  Importantly for us, one can show the limit in Definition \ref{asymp_con_defn} coincides with the $t = 0$ time slice of the Ricci flow induced by an asymptotically conical gradient shrinker.

  \begin{theorem}\label{cone_lower_bound}
    \sloppy For $n \geq 4$, let $\left(M,g,f\right)$ be an $n$-dimensional complete connected gradient Ricci shrinker with finitely many ends which is asymptotically conical to the cone $\left(\mathcal{C}\left(\Sigma\right), g_{\mathcal{C}} := dr^2 + r^2 g_\Sigma\right)$ with vertex $p_{\mathcal{C}}$, where the cone link $\left(\Sigma, g_\Sigma\right)$ is an $\left(n-1\right)$-dimensional closed manifold. Then there is a continuous function $f_{\mathcal{C}}: \mathcal{C}\left(\Sigma\right) \rightarrow \mathbb{R}$ so that if $L_{f_{\mathcal{C}}}: L^2_{f_{\mathcal{C}}}\left(\mathcal{C}\left(\Sigma\right)\right) \rightarrow L^2_{f_{\mathcal{C}}}\left(\mathcal{C}\left(\Sigma\right)\right)$ is upper semi-bounded we have

    \begin{equation*}
      \mathrm{Ind}_{f_\mathcal{C}}\left(\mathcal{C}\left(\Sigma\right)\right) \leq \mathrm{Ind}_f\left(M\right).    
    \end{equation*}
  \end{theorem}

  The intuition behind Theorem \ref{cone_lower_bound} is if one takes the blow-down limit of the shrinker $M$ (or flows the shrinker until it becomes singular) there should be a loss of topology, which should cause the resulting limiting space to be more stable. The assumption that $L_{f_{\mathcal{C}}}: L^2_{f_{\mathcal{C}}}\left(\mathcal{C}\left(\Sigma\right)\right) \rightarrow L^2_{f_{\mathcal{C}}}\left(\mathcal{C}\left(\Sigma\right)\right)$ is semi-bounded above is a technical condition. We comment more on this after the proof of Theorem \ref{cone_lower_bound}.

  Our strategy to prove the upper semi-continuity estimate in Theorem \ref{semi_continuity_thm} closely follows Workman's in \cite{workman_semicontinuity}, which in turn is a simplified version of the analysis in \cite{dlgr_upper_semicontinuity}. This simplification stems from the availability of $L^2$-Sobolev inequalities in dimensions $n \geq 3$. The main insight from \cite{dlgr_upper_semicontinuity} involved working with a \emph{weighted} eigenvalue problem. For us, this causes the eigenvalue problem to become scale invariant so it is preserved, in some sense, when one blows-up around an orbifold point. On the other hand, this approach comes with its own difficulties, as the weight function we introduce blows-up near orbifold points. To handle this, we will work with Lorentz spaces, a certain type of weak $L^p$-space. The proof then proceeds by studying the weighted problem on each bubble along the lines of the construction in \cite{by_bubbling} and then proving a spectral version of the ``neck theorem'' proved in Theorem $3.4$ of \cite{by_bubbling}. This latter part shows small annuli around each orbifold point cannot contribute any index or nullity. Finally, we show that the weighted and unweighted indexes and nullities coincide. As for the lower semi-continuity results, the proofs rely on cut-off functions and perturbation arguments, as well as the analysis in earlier sections.

  It is also worth noting that results analogous to Theorem \ref{semi_continuity_thm} have been proved in other places, for instance \cite{dalio_gianocca,gauvrit_laurain,hirsch_lamm, michelat,michelat_riviere}. Many of these papers adapt the strategy of \cite{dlgr_upper_semicontinuity}, while \cite{hirsch_lamm} uses a different method of proof and is of independent interest.

  The close resemblence of our proofs to those in \cite{workman_semicontinuity} speaks to the strength and robustness of the method, which can hopefully be used in other geometric settings. Aside from the scaling of the problem, the strategy hinges on the body and bubbles having Euclidean volume growth and admitting Euclidean type Sobolev inequalities, at least at appopriate scales. In this paper, we are able to show the analysis works if the stability operator acts on sections of some vector bundle and if the manifolds under consideration are non-compact, but come with a natural notion of finite weighted volume. We also refine the methods in \cite{workman_semicontinuity}, allowing us to show the weighted and unweighted nullities on the bubbles coincide. This is accomplished by appealing to results concerning operators acting as maps between Sobolev spaces weighted by some power of the distance function. We also outline how this can be adapted to the sequences of CMC hypersurfaces considered in \cite{workman_semicontinuity}.

  This paper is organized as follows. In Section \ref{preliminaries} we present some reults about Ricci shrinkers and then give an overview of orbifold convergence and the bubble tree construction from \cite{by_bubbling}. Section \ref{functional} is devoted to presenting various functional analytic and spectral properties of the stability operators we consider. In particular, this will tell us that the nullity and index (on the shrinkers and bubbles) are finite when working in appropriate (weighted) Sobolev spaces. In Section \ref{upper_semi_cont}, the upper semi-continuity estimate is proved after introducing Lorentz spaces and formulating the weighted eigenvalue problem. Finally, we prove the lower semi-continuity estimate and Theorem \ref{cone_lower_bound} in Section \ref{lower_proof}.\\

  \textbf{Acknowledgements.} The author is grateful to Myles Workman for various discussions about his article \cite{workman_semicontinuity}, which were of great aid while working on this paper. The author also thanks Klaus Kr\"oncke for many helpful conversations.
  
\section{Preliminaries}\label{preliminaries}
    \subsection{Some Identities and Results for Gradient Shrinkers}\label{shrinker_prelims}
        We now recall some important identities and results which hold on a gradient Ricci shrinker $\left(M,g,f\right)$. We will use $p$ to denote a minimum of the potential function $f$ throughout this section and the rest of the paper. The following are consequences of \eqref{grs_eqn}:   

        \begin{align}
            R_g + \Delta f &= \frac{n}{2} \label{grs_traced}\\
            R_g + \left|\nabla f\right|^2 &= f - \mu\left(g\right). \label{grs_auxiliary}
        \end{align}

        Taking the trace of \eqref{grs_eqn} yields \eqref{grs_traced}, while a proof of \eqref{grs_auxiliary} can be found in Chapter 1 of \cite{rf_tech_appsI}. Moreover, the scalar curvature of gradient shrinkers is always non-negative: $R_g \geq 0$. See \cite{grs_scal_curv} for a proof using a maximum principle argument. In fact, $R_g > 0$ unless the shrinker is flat, which follows from considering the evolution equation for $R_g$ on a shrinker. One can also show the potential function $f$ grows at most quadratically:

        \begin{equation}\label{potential_growth}
            \frac{1}{4}\left(d\left(x,p\right)-5n\right)^2_+ \leq f\left(x\right) - \mu\left(g\right) \leq \frac{1}{4}\left(d\left(x,p\right)+\sqrt{2n}\right)^2.
        \end{equation}

        For a proof, we refer the reader to Lemma $2.1$ in \cite{HM1}. Together with $R_g > 0$ and \eqref{grs_auxiliary} one obtains $\inf_M f - \mu\left(g\right) > 0$. We can then normalize the following weighted volume, possibly after modifying $f$ by a constant:

        \begin{equation}\label{grs_normalization}
            \int_M \left(4\pi\right)^{-\frac{n}{2}}e^{-f}dV_g = 1.
        \end{equation}

        Moving forward, we will implicitly assume our shrinkers are normalized so that \eqref{grs_normalization} always holds. The natural second order self-adjoint elliptic operator associated to this weighted volume measure is the drift Laplacian $\Delta_f$. Moreover, \eqref{grs_normalization} gives the space of Ricci shrinkers a notion of unit volume. Since gradient Ricci shrinkers are a generalization of positive Einstein manifolds, one could then expect that Theorem \ref{hm_compactness_thm} and the results in \cite{by_bubbling} are true, since analogous statements are known to hold for Einstein manifolds. On the other hand, the normalization \eqref{grs_normalization} has the added benefit of ensuring certain weighted Sobolev spaces obey compact embedding relations as if $M$ was a closed manifold. See Section \ref{grs_functional} for more details and the precise results.

        This all allows us to ensure Perelman's $\mathcal{W}$-entropy

        \begin{equation*}
            \mathcal{W}\left(g,f,\tau\right) := \left(4\pi \tau\right)^{-\frac{n}{2}}\int_M \left(\tau \left(\left|\nabla f\right|^2 + R_g\right) + f - n\right) e^{-f} dV_g
        \end{equation*}

        is well-defined on a shrinker. This uses $R_g > 0$, \eqref{grs_auxiliary}, and \eqref{potential_growth} to deduce that $f,\left|\nabla f\right|^2$, and $R_g$ grow at most quadratically. A lower bound follows from the non-negativity of the quantities just discussed as well as the normalization \eqref{grs_normalization}. The entropy is then defined as

        \begin{equation}\label{nu_entropy}
            \mu\left(g\right) := \inf \left\{ \mathcal{W}\left(g,f,\tau\right) : \tau > 0, f \in C^\infty\left(M\right)~\mathrm{such~that}~\left(4\pi \tau\right)^{-\frac{n}{2}}\int_M e^{-f}dV_g = 1\right\}.
        \end{equation}

        By work of Carillo--Ni (\cite{carillo_ni}), the infimum is always achieved when $g$ is a shrinker metric. The entropies are invariant under rescalings and pulling back by a diffeomorphism. For further properties of the entropy on shrinkers we refer the reader to work of Li--Wang (Section $4$ in \cite{heat_kernel_grs} in particular).

        Furthermore, one can prove a Euclidean upper bound for the volume of balls centered at $p$ (Lemma $2.2$ in \cite{HM1}):

        \begin{equation}\label{grs_upper_volume}
            \mathrm{Vol}_g\left(B_g\left(p,r\right)\right) \leq V_0 r^n, \quad \forall r>0
        \end{equation}

        where $V_0 > 0$ is a constant depending only on $n$. A local volume non-collapsing result also holds when $\mu\left(g\right) \geq \underline{\mu} > -\infty$ (Lemma $2.3$ in \cite{HM1}):

        \begin{equation}\label{grs_lower_volume}
            \mathrm{Vol}_g\left(B_g\left(q,\delta\right)\right) \geq v_0 \delta^n,
        \end{equation}

        for any $B_g\left(q,\delta\right)\subset B_g\left(p,r\right)$ with $0 < \delta \leq 1$ and $v_0 > 0$ a constant depending only on $n, \underline{\mu}, r$.

        Finally, we have the following $\varepsilon$-regularity result (Lemma $3.3$ in \cite{HM1}) which will be useful later on:

        \begin{lemma}\label{hm_eps_reg}
          Let $\left(M,g,f\right)$ be a gradient Ricci shrinker with normalization \eqref{grs_normalization} and $p$ a minimum of the potential function $f$. Then for every $\ell \in \mathbb{N}$ and $r > 0$, there is an $\varepsilon_{\mathrm{reg}} = \varepsilon_{\mathrm{reg}}\left(r,n,\underline{\mu}\right) > 0$, a $K = K\left(\ell, r, n, \underline{\mu}\right)$, and a $\delta_0 = \delta_0\left(r,n,\underline{\mu}\right) > 0$ such that for every $B_g\left(x,\delta\right) \subset B_g\left(p,r\right)$ with $\delta \in \left(0,\delta_0\right]$ we have the implication

          \begin{equation*}
            \left|\left|\mathrm{Rm}_g\right|\right|_{L^{\frac{n}{2}}\left(B_g\left(x,\delta\right)\right)} \leq \varepsilon_{\mathrm{reg}} \implies \sup_{B_g\left(x,\frac{\delta}{4}\right)} \left|\nabla^\ell \mathrm{Rm}_g\right| \leq \frac{K}{\delta^{2+\ell}} \left|\left|\mathrm{Rm}_g\right|\right|_{L^{\frac{n}{2}}\left(B_g\left(x,\delta\right)\right)}.
          \end{equation*}
        \end{lemma}

        We note in particular that, when considering a sequence of gradient shrinkers, as long as one fixes a uniform radius $r$, then the constants $\varepsilon_{\mathrm{reg}}, K, \delta_0$ will all be uniform along the sequence. We also note here that it is possible to prove a similar result without $\varepsilon_{\mathrm{reg}}$ depending on $r$. See, for instance, Theorem $1.8$ in \cite{eps_reg_wang_wang} and the associated proof.

    \subsection{Bubble Tree Convergence}\label{bubble_tree_prelims}
        In future sections the analytic aspects might obscure a lot of the geometry, so we now take the opportunity to outline the bubble tree construction from Section $5$ of \cite{by_bubbling}. Moreover, Theorem \ref{semi_continuity_thm} will be proved by analyzing a weighted eigenvalue problem at each step of the construction, so it is worth taking some time to familiarize ourselves with the procedure. We also note that bubble tree convergence is known in a variety of other geometric settings and we refer the reader to the references in \cite{by_bubbling} and Section \ref{introduction} for more information.

        We now recall what is meant by (pointed) Gromov--Hausdorff and (pointed) orbifold Cheeger--Gromov convergence. 

        \begin{definition}\label{GHA}
            A pointed map $f: \left(X,p\right) \to \left(Y,q\right)$ between two metric spaces $\left(X, d_X, p\right)$, $\left(Y,d_Y,q\right)$ is an \emph{$\varepsilon$-pointed Gromov--Hausdorff approximation} ($\varepsilon$-PGHA) if it is almost an isometry and almost onto in the following sense:

            \begin{enumerate}
                \item $\left|d_X\left(x_1, x_2\right) - d_Y\left(f\left(x_1\right), f\left(x_2\right)\right)\right| \leq \varepsilon$, for all $x_1,x_2 \in B_{d_X}\left(p,\frac{1}{\varepsilon}\right)$,
                \item For all $y \in B_{d_Y}\left(q,\frac{1}{\varepsilon}\right)$ there exists $x\in B_{d_X}\left(p,\frac{1}{\varepsilon}\right)$ with $d_Y\left(y,f\left(x\right)\right)\leq \varepsilon$.
            \end{enumerate}

            We say $\left(X_i,p_i\right) \to \left(Y,q\right)$ as $i\to\infty$ in the \emph{pointed Gromov--Hausdorff sense} if, as $i \rightarrow \infty$,

            \begin{equation*}
                \inf \left\{\varepsilon > 0 : \exists \text{ $\varepsilon$-pGHA } f_1:\left(X_i,p_i\right) \to \left(Y,q\right) \text{ and } f_2:\left(Y,q\right) \to \left(X_i,p_i\right)\right\} \to 0.
            \end{equation*}
        \end{definition}

        One can think of Gromov--Hausdorff convergence as a geometric notion of $C^0$-convergence. In particular, it allows one to still make some sense of sequences which have singular limits, such as an orbifold shrinker, which we now define.

        \begin{definition}\label{defn_orb_grs}
            A complete metric space $M_\infty$ is called an \emph{orbifold Ricci shrinker} if it is a smooth Ricci shrinker away from a locally finite set $\mathcal{Q}$ of singular points and at every $q \in \mathcal{Q}$, $M_\infty$ is modeled on $\mathbb{R}^n/\Gamma$ for some finite group $\Gamma \subset O\left(n\right)$. Moreover, there exists an associated covering $\mathbb{R}^n \supset B\left(0,\varrho\right)\backslash \left\{0\right\} \stackrel{\psi}{\to} U \backslash \left\{q\right\}$ of some neighborhood $U\subset M_\infty$ of $q$ such that $\left(\psi^\ast g_\infty,\psi^\ast f_\infty\right)$ can be smoothly extended to a gradient shrinker over the origin.
        \end{definition}

        Now we can introduce the notion of orbifold Cheeger--Gromov convergence, which can be thought of as a geometric version of almost everywhere $C^\infty$-convergence and $C^0$-convergence everywhere. If no orbifold points form in the limit then this reduces to the usual definition of pointed smooth Cheeger--Gromov convergence.

        \begin{definition}\label{defn_orb_grs_conv}
            A sequence of gradient shrinkers $\left(M_i, g_i, f_i, p_i\right)$ converges to an orbifold gradient shrinker $\left(M_\infty, g_\infty, f_\infty, p_\infty\right)$ in the \emph{pointed orbifold Cheeger--Gromov sense} if the following properties hold:

            \begin{enumerate}
                \item There exist a locally finite set $\mathcal{Q} \subset M_\infty$, an exhaustion of $M_\infty \backslash \mathcal{Q}$ by open sets $U_i$, and smooth embeddings $\varphi_i: U_i \to M_i$ such that $\left(\varphi^\ast_i g_i, \varphi^\ast_i f_i\right)$ converges to $\left(g_\infty, f_\infty\right)$ in the $C^\infty_{\mathrm{loc}}$-sense on $M_\infty \backslash \mathcal{Q}$.
                \item Each of the above maps $\varphi_i$ can be extended to an $\varepsilon$-pGHA which yields a convergent sequence $\left(M_i, d_i, p_i\right) \to \left(M_\infty, d_\infty, p_\infty\right)$ in the pointed Gromov--Hausdorff sense.
            \end{enumerate}
        \end{definition}

        For a more detailed discussion and how one can prove such convergence for gradient shrinkers, we refer the reader to Sections $2-3$ of \cite{HM1}. If one disregards the points about the potential function, Definition \ref{defn_orb_grs} and Definition \ref{defn_orb_grs_conv} carry over to a general sequence of manifolds converging to an orbifold. It is also worth noting that, while $\mathcal{Q}$ consists of singular points, they are relatively well-behaved (or ``mild'') since the existence of the covering map $\psi$ from Definition \ref{defn_orb_grs} allows one to locally work in a smooth covering space and then project back to the orbifold. This allows for the definition of H\"older and Sobolev spaces on orbifolds using the covering map $\psi$ and a partition of unity to construct ``orbifold charts". Note this also yields a way to formulate integration on orbifolds. We refer the reader to \cite{farsi} and the references therein for more information.

        \begin{remark}\label{orbifold_remark}
            The orbifolds we will encounter are ``orbifolds with bounded curvature'' due to our assumed uniform curvature bound on $M_i \backslash K_i$ and Step $3$ of the proof of Theorem \ref{hm_compactness_thm} in \cite{HM1}. In particular, each point on the orbifold $M_\infty$ can be locally covered by a manifold with bounded curvature. This makes applying PDE techniques in later sections much easier in principle when we work directly on the orbifold, rather than the sequence of shrinkers converging to it. However, we will try to avoid using the precise orbifold structure as much as possible and endeavor to treat the orbifold points as general conical singularities, which will hopefully be of future use when dealing with such structures.
        \end{remark}

        As a final bit of preparation before discussing the bubble tree construction, we provide the definition of an ALE bubble:

        \begin{definition}\label{ale_bubble}
            An $n$-dimensional manifold (or an orbifold with finitely many orbifold points) $\left(V, h\right)$ with one end is \emph{asymptotically locally Euclidean (ALE) of order $\tau > 0$} if there is a compact set $K \subset V$, a constant $R > 0$, a finite group $\Gamma \subset O\left(n\right)$ acting freely on $\mathbb{R}^n \backslash B\left(0,R\right)$, as well as a smooth diffeomorphism $\psi: V \backslash K \to \left(\mathbb{R}^n \backslash B\left(0,R\right)\right)/\Gamma$ such that

            \begin{align*}
                \left(\varphi^\ast h\right)_{ij}\left(x\right) &= \delta_{ij} + O\left(\left|x\right|^{-\tau}\right),\\
                \partial^k \left(\varphi^\ast h\right)_{ij}\left(x\right) &= O \left(\left|x\right|^{-\tau-k}\right)
            \end{align*}

            for all $k \geq 1$ and $x,y \in \mathbb{R}^n \backslash B\left(0,R\right)$. Here $\varphi := \psi^{-1} \circ \pi$ where $\pi: \mathbb{R}^n \to \mathbb{R}^n/\Gamma$ is the natural projection. We say that an $n$-dimensional manifold (or orbifold with finitely many orbifold points) is an \emph{ALE bubble} if it is complete and non-compact with one end, Ricci-flat, non-flat with bounded $L^{\frac{n}{2}}$-Riemannian curvature, and ALE of order $n$.
        \end{definition}

        Note that this differs slightly from the definition in \cite{by_bubbling}. This used an older work on coordinates at infinity for ALE bubbles. In this paper, we use a more recent result due to Kr\"oncke--Szab\'o in \cite{kroncke_szabo} which tells us the bubbles are of order $n$ in general.

        Now for the bubble tree construction. This is all encapsulated in Theorem $1.2$ in \cite{by_bubbling} and the full construction can be found in Section $5$ of \cite{by_bubbling}. Consider a pointed sequence of shrinkers $\left(M_i,g_i,f_i,p_i\right)$ as in Theorem \ref{hm_compactness_thm} which converges in the pointed orbifold Cheeger--Gromov sense to an orbifold shrinker $\left(M_\infty, g_\infty, f_\infty, p_\infty\right)$ with singular set $\mathcal{Q}$. The convergence on $M_\infty \backslash \mathcal{Q}$, which we call the \textit{body region}, is locally smooth. Using Lemma \ref{hm_eps_reg}, one can show the orbifold points form due to $L^{\frac{n}{2}}$-curvature (``energy'') concentrating in small regions. Note that this also means orbifold point formation implies loss of energy in the limit and, by the Chern--Gauss--Bonnet theorem, a potential loss of topology. This characterization leads to, for each $q \in \mathcal{Q}$, a collection of \emph{point-scale sequences} $\left\{\left(q^k_i, s^k_i\right)\right\}^{N_q}_{k=1}$ such that, for every $k = 1, \dots, N_q$, 

        \begin{align*}
            s^1_i \leq s^2_i &\leq \dots \leq s^{N_q}_i,\\
            s^k_i &\rightarrow 0,\\
            q^k_i &\rightarrow q \in \mathcal{Q}.
        \end{align*}

        These point-scale sequences are determined by looking in small balls along the sequence which contain more than a critical amount of energy. This also means the scales $s^k_i$ represent how fast curvature is blowing up around $q^k_i$, making them a natural scaling factor to use in our blow-up analysis. We thus first consider the rescaled metric $\widetilde{g}^1_i := \left(s^1_i\right)^{-2} g_i$ and the associated point $q^1_i$ around which the energy coalesces the fastest. Then, by Theorem $2.6$ in \cite{by_bubbling}, we can pass to a subsequence to see that

        \begin{equation*}
            \left(M_i, \widetilde{g}^1_i, q^1_i\right) \rightarrow \left(V^1, h^1, q^1_\infty\right)
        \end{equation*}
    
        in the pointed orbifold Cheeger--Gromov sense. On the other hand, we blew-up around the fastest forming orbifold point, so no other orbifold points can be present on $V^1$. $V^1$ is thus a smooth manifold and the convergence above is actually in the smooth pointed Cheeger--Gromov sense. One can also show $V^1$ is non-flat, Ricci-flat, has a single end, and is ALE of order $n$. This follows from work in \cite{kroncke_szabo} and how various curvature quantities behave under rescalings of the metric. We call such smooth limits \emph{leaf bubbles}.

        On the other hand, multiple orbifold points could form and in such a case we consider the next scale $s^2_i$ and associated point $q^2_i$. We then try to repeat the procedure from before by considering $\widetilde{g}^2_i := \left(s^2_i\right)^{-2}g_i$ and using Theorem $2.6$ in \cite{by_bubbling} to conclude that, after passing to a subsequence,

        \begin{equation*}
            \left(M_i, \widetilde{g}^2_i, q^2_i\right) \rightarrow \left(V^2, h^2, q^2_\infty\right)
        \end{equation*}

        in the pointed orbifold Cheeger--Gromov sense. $V^2$ again has a single end and is non-flat, Ricci-flat, and ALE of order $n$, but we need to check whether or not $V^2$ is smooth. To do this, we first note that, similarly to Claim $5.2$ in \cite{by_bubbling}, one has

        \begin{equation}\label{rescaled_distance}
            \frac{s^1_i}{s^2_i} + \frac{s^2_i}{s^1_i} + \frac{d_{g_i}\left(q^1_i,q^2_i\right)}{s^2_i} \rightarrow \infty
        \end{equation}

        as $i \rightarrow \infty$, and likewise if one swaps the roles of the bubble scales. In general, this holds for any two distinct point-scale sequences $\left(q^k_i, s^k_i\right)$ and $\left(q^\ell_i, s^\ell_i\right)$. Importantly, \eqref{rescaled_distance} tells us one of the following occurs:

        \begin{align}
            \frac{d_{g_i}\left(q^1_i,q^2_i\right)}{s^2_i} &\rightarrow \infty \label{separable},\\
            \frac{d_{g_i}\left(q^1_i,q^2_i\right)}{s^2_i} &\leq M \label{not_separable}
        \end{align}

        for some $M \in \mathbb{R}_+$. We say the bubbles are \emph{separable} if \eqref{separable} still holds when $s^2_i$ is replaced by $s^1_i$. In general, two distinct bubbles associated to the point-scale sequences $\left(q^k_i, s^k_i\right)$ and $\left(q^\ell_i, s^\ell_i\right)$ are separable if, in the above discussion, after replacing $q^1_i,q^2_i$ with $q^k_i,q^\ell_i$ respectively, \eqref{separable} holds when replacing $s^1_i$ with $s^k_i$ and $s^2_i$ by $s^\ell_i$. Note also that if $\frac{s^k_i}{s^\ell_i} + \frac{s^\ell_i}{s^k_i}$ remains bounded then \eqref{rescaled_distance} tells us the bubbles $\left(V^k,h^k\right)$ and $\left(V^\ell, h^\ell\right)$ are separable.

        Due to how the distance function behaves under rescalings, \eqref{separable} implies $q^1_\infty \not\in V^2$, while \eqref{not_separable} tells us $q^1_\infty \in V^2$. The former means $\left(V^2, h^2\right)$ is smooth, hence a leaf bubble, while the latter implies $q^1_\infty$ is an orbifold point on $V^2$, since the distance between $q^1_i$ and $q^2_i$ remains bounded and $\frac{s^1_i}{s^2_i} \rightarrow 0$. We call such singular bubbles \emph{intermediate bubbles}. These are the main source of technical difficuties, which we will describe in more detail shortly. The above process repeats until every point-scale sequence associated to each $q \in \mathcal{Q}$ has been exhausted. This is always guaranteed, provided one works in a ball of radius $r > 0$. In particular, it can be shown that $\left|\mathcal{Q}_r\right| := \left|\mathcal{Q} \cap B_{g_\infty}\left(p_\infty,r\right)\right| \leq \frac{E\left(r\right)}{\varepsilon_{\mathrm{reg}}\left(r\right)} < \infty$, hence $\sum_{q \in \mathcal{Q}_r} N_q < \infty$. Here $\varepsilon_{\mathrm{reg}}\left(r\right)$ is the $\varepsilon$-regularity constant from Lemma \ref{hm_eps_reg}. Once all the point-scale sequences have been exhausted, blowing up at some other scale $\rho_i \rightarrow 0$ such that \eqref{rescaled_distance} is satisfied yields a flat limit.

        The technical difficulties introduced by intermediate bubbles have to do with the need to account for \emph{neck regions}. These are closed annuli 

        \begin{equation*}
            A^{g_i}_{s',s}\left(q^k_i\right) := \overline{B_{g_i}\left(q^k_i,s\right)} \backslash B_{g_i}\left(q^k_i,s'\right)
        \end{equation*}

        for $s' < s$ and $q^k_i \rightarrow q \in \mathcal{Q}$. When $s',s \ll 1$ one can show that $A^{g_i}_{s',s}\left(q^k_i\right)$ has a single connected component intersecting $\partial B_{g_i}\left(q^k_i, s\right)$ (Lemma $3.1$ in \cite{by_bubbling}). Moving forward we will abuse notation slightly and write $A^{g_i}_{s',s}\left(q^k_i\right)$ to denote this connected component. This is also what allows us to guarantee the bubbles each have a single end. Neck regions arise, for instance, when proving the following energy identity from \cite{by_bubbling}, which holds for any $r \geq 2$ such that $\mathcal{Q} \cap \partial B_{g_\infty}\left(p_\infty,r\right) = \emptyset$:

        \begin{equation*}
            \lim_{i \to \infty} \int_{B_{g_i}\left(p_i,r\right)} \left|\mathrm{Rm}_{g_i}\right|^{\frac{n}{2}}_{g_i} dV_{g_i} = \int_{B_{g_\infty}\left(p_\infty,r\right)} \left|\mathrm{Rm}_{g_\infty}\right|^{\frac{n}{2}}_{g_\infty} dV_{g_\infty} + \sum_{q \in \mathcal{Q}_r} \sum_{k = 1}^{N_q} \int_{V^k}\left|\mathrm{Rm}_{h^k}\right|^{\frac{n}{2}}_{h^k} dV_{h^k}.
        \end{equation*}

        The proof mainly involves looking at how the energy behaves when blowing-up around each orbifold point. If we end up with a leaf bubble, say the one associated to the point-scale sequence $\left(q^1_i,s^1_i\right)$, then we can use the scale invariance of the energy to compute as follows:

        \begin{equation*}
            \lim \limits_{i,R \rightarrow \infty} \int_{B_{g_i}\left(q^1_i, Rs^1_i\right)} \left|\mathrm{Rm}_{g_i}\right|^{\frac{n}{2}}_{g_i} dV_{g_i} = \lim \limits_{i,R \rightarrow \infty} \int_{B_{\widetilde{g}^1_i}\left(q^1_i, R\right)} \left|\mathrm{Rm}_{\widetilde{g}^1_i}\right|^{\frac{n}{2}}_{\widetilde{g}^1_i} dV_{\widetilde{g}^1_i} = \int_{V^1} \left|\mathrm{Rm}_{h^1}\right|^{\frac{n}{2}}_{h^1} dV_{h^1}.
        \end{equation*}

        For an intermediate bubble, say the one associated to the point-scale sequence $\left(q^2_i,s^2_i\right)$ such that $d_{g_i}\left(q^1_i,q^2_i\right)\left(s^2_i\right)^{-1} \leq M$, we know we will still see $q^1_\infty$ after blowing-up around $q^2_i$. This motivates the following decomposition of the bubble region:

        \begin{equation*}
            B_{g_i}\left(q^2_i,Rs^2_i\right) = \left(B_{g_i}\left(q^2_i,Rs^2_i\right) \backslash B_{g_i}\left(q^1_i,\frac{1}{R} s^2_i\right)\right) \cup A^{g_i}_{Rs^1_i, \frac{1}{R}s^2_i}\left(q^1_i\right) \cup B_{g_i}\left(q^1_i,Rs^1_i\right).
        \end{equation*}

        For the first and third regions on the right hand side, the computation to show the energy is recovered is essentially the same as in the leaf bubble case. In particular, the regions each involve only a single bubble scale, which allows a rescaling argument to work. On the other hand, the second region (the neck region) involves two distinct bubble scales, so merely rescaling does not yield the desired result. To overcome this, a ``neck theorem'' is typically proved to get some extra information about the structure of the necks. In our setting we have Theorem $3.4$ from \cite{by_bubbling}, which tells us:

        \begin{itemize}
            \item \sloppy For $s' < s \ll 1$, each neck $A^{g_i}_{s',s}\left(q^k_i\right)$ is diffeomorphic to an annulus on a Euclidean cone: $\mathcal{C}_{s',s}\left(S^{n-1}/\Gamma^k\right) \cong \left[s',s\right] \times S^{n-1}/\Gamma^k$. Here $\Gamma^k \subset O\left(n\right)$ is a finite isometry group.
            \item The induced metric on the subannulus $A^{g_i}_{s'',2s''}\left(q^k_i\right) \subset A^{g_i}_{s',s}\left(q^k_i\right)$, can be made arbitrarily close, after rescaling, to the flat Euclidean metric.
        \end{itemize} 

        This, in addition to some refined energy estimates and an improved Kato inequality, yielded that no energy concentrates in the neck regions. An induction argument then completes the bubble tree construction and the proof of the energy identity.

        While neck regions cause a lot of technical difficulties, they have a nice geometric interpretation. This is encapsulated in the following formal definition of a bubble tree:

        \begin{definition}\label{bubble_tree_definition}
            A \emph{bubble tree} $\mathcal{T}$ is a tree whose vertices are ALE bubbles and whose edges are neck regions. The single ALE end of each vertex is connected by a neck region (which it meets at its smaller boundary component) to its parent and possibly further ancestors toward the \emph{root bubble} of the tree $\mathcal{T}$, while at possibly finitely many isolated orbifold points it is connected by more necks (which it meets at their larger boundary components) to its children and possibly further descendants toward leaf bubbles of $\mathcal{T}$. We say two bubble trees $\mathcal{T}_1$ and $\mathcal{T}_2$ are separable if their root bubbles are separable.
        \end{definition}

        We end this section by explaining the reasoning behind assuming $R_{g_i} \leq R_0$ and the number of bubbles forming is globally finite in Theorem \ref{semi_continuity_thm}. In short, they guarantee the bubble tree construction terminates if we consider the entire manifold. We now outline why this is and omit the $i$ index for simplicity. When $R_g \leq R_0$, one can use \eqref{grs_auxiliary} and \eqref{potential_growth} to show that

        \begin{equation*}
            \left|\nabla f\right|^2 \geq \frac{1}{4}\left(d_g\left(x,p\right) - 5n\right)^2_+ - R_0 + \mu\left(g\right).
        \end{equation*}

        This implies that the set of critical points of $f$, say $\mathrm{Crit}\left(f\right)$, is contained in the interior of a compact set, say $B_g\left(p,r_{\mathrm{orb}}\right)$, where $r_{\mathrm{orb}}$ is a positive constant depending only on $R_0, n,\underline{\mu}$. Possibly after increasing $r_{\mathrm{orb}}$, from now on we will take $r_{\mathrm{orb}} \geq 2$. One can also show that every orbifold point is a critical point of $f$. More specifically, for the finite isometry group $\Gamma \subset O\left(n\right)$ associated to $q \in \mathcal{Q}$, let $\psi : \mathbb{R}^n \rightarrow U$ be the orbifold chart arising from the covering map in Definition \ref{defn_orb_grs} with $U$ a neighborhood of $q$. Then $\nabla^{\psi^\ast g} \psi^\ast f\left(q\right)$ is fixed by every $\gamma \in \Gamma$. Thus $\nabla^{\psi^\ast g} \psi^\ast f\left(q\right)$ is the zero vector and $\mathcal{Q} \subset \mathrm{Crit}\left(f\right)$.

        While this tells us where the orbifold points are located, we would still like to bound their number and the number of nodes in the resulting bubble trees uniformly. Ideally, one would enlarge the radius of $B_g\left(p,r_\mathrm{orb}\right)$ to capture all the Betti numbers. Once this happens one could appeal to, say, the Chern--Gauss--Bonnet theorem for Ricci-flat ALE spaces to deduce no more bubbles can form. This at first seems feasible, as work in \cite{grs_topology} tells us the Betti numbers of a gradient Ricci shrinker with bounded scalar curvature are all finite. However, as noted in \cite{wright_thesis}, quotients of Gibbons--Hawking spaces $V$ with Betti numbers $b_i\left(V\right) = 0$ for $i = 1,\dots, n$ can potentially arise (at least in dimension $n = 4$), the easiest example being the $\mathbb{Z}_2$-quotient of the Eguchi--Hanson metric. This means that one can continue forming bubbles by capturing more and more energy, even after all the Betti numbers have been accounted for. This is all to say we assume the number of bubbles forming is globally finite to avoid such technical difficulties.

        It is also worth noting that the above discussion is only important if one wants to prove \emph{global} results. If one is instead interested in local results or closed shrinkers, then the assumptions on the scalar curvature  and global finiteness of $\mathcal{Q}$ can be relaxed.
\section{Properties of the Stability Operators}\label{functional}
    We now turn to presenting various technical results that we will have need of throughout the rest of this paper. These largely concern functional analytic and spectral properties of the drift Einstein operator $L_f u := \Delta_f u + 2\mathrm{Rm}_g \ast u$ on gradient shrinkers $\left(M,g,f\right)$ and the Einstein operator $Lu := \Delta u + 2\mathrm{Rm}_h \ast u$ on ALE bubbles $\left(V,h\right)$. The results we will discuss involve various weighted Sobolev spaces involving either the weighted volume measure $e^{-f}dV_g$ (on shrinkers), or some polynomial weight $\rho^{-k}$ (on the bubbles), where $\rho$ is essentially the distance from a fixed point $q$.

    A couple remarks before moving on. If we need to emphasize the dependence of $L_f$ or $L$ on a certain metric, we will write $L_{g,f}$ or $L_h$. We will usually omit this to ease notation unless it is not clear from context. Also, as mentioned earlier, the definitions of the various spaces carry over to the orbifold setting after passing to local orbifold covers.

    \subsection{The Drift Einstein Operator on Shrinkers}\label{grs_functional}
        As mentioned in Section \ref{introduction}, an advantage of the normalization \eqref{grs_normalization} is that the space of gradient shrinkers behaves similarly to the space of positive Einstein metrics with unit volume. In this section, this manifests in certain weighted Sobolev spaces on shrinkers having analytic properties similar to those of the unweighted spaces on closed manifolds. The results are taken from Section $2$ of \cite{stolarski_grs}, where an interested reader can find the proofs.

        \begin{definition}\label{wgtd_sobolev}
            Let $\left(M,g,f\right)$ be an $n$-dimensional complete connected gradient Ricci shrinker. Further let $E$ be a vector bundle on $\Omega \subseteq M$, assume $g$ induces a metric on $E$, and denote the set of smooth sections of $E$ by $\Gamma\left(\Omega,E\right)$. Then the weighted space $L^2_f\left(\Omega,E\right)$ is the completion of

            \begin{equation*}
                \left\{u \in \Gamma\left(\Omega,E\right) : \int_\Omega \left|u\right|^2 e^{-f} dV_g < \infty\right\}
            \end{equation*}

            with respect to the inner product

            \begin{equation*}
                \left<u_1,u_2\right>_{L^2_f\left(\Omega\right)} = \int_\Omega \left<u_1,u_2\right> e^{-f} dV_g.
            \end{equation*}

            For any $k \in \mathbb{N}$, $W^{k,2}_f\left(\Omega,E\right)$ is the completion of

            \begin{equation*}
                \left\{u \in \Gamma\left(\Omega,E\right) : \sum^k_{\ell = 0} \int_\Omega \left|\nabla^\ell u\right|^2 e^{-f} dV_g < \infty\right\}
            \end{equation*}

            with respect to the inner product

            \begin{equation*}
                \left<u_1,u_2\right>_{W^{k,2}_f\left(\Omega\right)} = \sum^k_{\ell = 0} \int_\Omega \left<\nabla^\ell u_1, \nabla^\ell u_2\right> e^{-f} dV_g.
            \end{equation*}
        \end{definition}

        Note that this definition tells us $L^2_f\left(\Omega,E\right) = W^{0,2}_f\left(\Omega,E\right)$. Moving forward, we will mainly consider when $\Omega = M$, $\Omega = B_g\left(p,R\right)$ for some radius $R > 0$, or a ball minus some other subset(s). We will also primarily use $E = S^2 T^\ast M$ and suppress writing the vector bundle $E$ unless it is not clear from context.

        As mentioned earlier, certain embeddings of our weighted Sobolev spaces are compact:

        \begin{proposition}\label{weighted_sobolev_compact}
            Let $\left(M,g,f\right)$ be an $n$-dimensional complete connected gradient Ricci shrinker with bounded scalar curvature. Then, for $\Omega \subseteq M$ and $k \geq 1$, the continuous embedding

            \begin{equation*}
                W^{k,2}_f\left(\Omega\right) \hookrightarrow W^{k-1,2}_f\left(\Omega\right)
            \end{equation*}

            is compact.
        \end{proposition}

        The proof for $k = 1$ is as in \cite{stolarski_grs} while the result for higher $k$ follows from an induction/iteration argument.

        Set 

        \begin{equation*}
            D\left(\Delta_f\right) := \left\{u \in W^{1,2}_f\left(M\right) : \Delta_f u \in L^2_f\left(M\right)\right\},
        \end{equation*}

        where $\Delta_f u$ is understood in the distributional sense. When $M$ has bounded curvature, one can show that $L_f: L^2_f\left(M\right) \rightarrow L^2_f\left(M\right)$ is semi-bounded above. Then, after modifying it by a multiple of the identity operator to $\widehat{L}_f$, say, one can show $\widehat{L}_f$ has a self-adjoint extension to $\widehat{L}_f : D\left(\Delta_f\right) \rightarrow L^2_f\left(M\right)$ (see page $330$, Section $124$ in \cite{riesz_nagy}). Then, because of Proposition \ref{weighted_sobolev_compact}, it can be shown that the associated resolvent operator is compact. This all allows for the application of Theorem $XIII.64$ in \cite{reed_simon}, which yields the following result:

        \begin{theorem}\label{spectral_theorem}
            Let $\left(M,g,f\right)$ be an $n$-dimensional complete connected gradient Ricci shrinker with bounded curvature. Then the following hold:

            \begin{enumerate}
                \item There is an orthonormal basis $\left\{u_j\right\}^\infty_{j=1}$ of $L^2_f\left(M\right)$ such that $u_j \in D\left(\Delta_f\right)$ is an eigenmode of $L_f$ with eigenvalue $\lambda_j \in \mathbb{R}$.

                \item The eigenvalues $\left\{\lambda_j\right\}^\infty_{j=1}$ satisfy $\lambda_j \leq \lambda_{j+1}$. 

                \item The eigenvalues $\lambda_j$ each have finite multiplicity, are given by the min-max principle, and $\lambda_j \rightarrow \infty$ as $j \rightarrow \infty$. Furthermore $\sigma\left(L_f\right)$, the spectrum of $L_f$, equals $\left\{\lambda_j\right\}^\infty_{j=1}$.
            \end{enumerate}
        \end{theorem}

        Note that the theorem is stated in terms of $L_f$, not $\widehat{L}_f$. Since $\widehat{L}_f$ differs from $L_f$ by a multiple of the identity operator we can pass from one to the other at the cost of shifting the eigenvalues, which causes no issues. Also, the way we have stated Theorem \ref{spectral_theorem} implicitly uses the eigenvalue equation 

        \begin{equation}\label{eigenvalue_strong}
            L_f u + \lambda u = 0.
        \end{equation}

        This means that, while we use the so called ``analyst's sign convention" for the Laplacian, we will end up counting negative eigenvalues when studying the Morse index of Ricci shrinkers.

        To account for the formation of conical singularities we need to formulate \eqref{eigenvalue_strong} in a suitable weak sense. We thus define the following bi-linear form:

        \begin{equation*}
            \mathcal{B}_f\left[u,\varphi\right] := \int_M \left(\left<\nabla u, \nabla \varphi\right> - 2\left<\mathrm{Rm}_g \ast u, \varphi\right>\right) e^{-f} dV_g
        \end{equation*}

        for $u, \varphi \in W^{1,2}_f\left(M\right)$. Then the weak version of \eqref{eigenvalue_strong} is

        \begin{equation}\label{eigenvalue_weak}
            \mathcal{B}_f\left[u,\varphi\right] = \lambda \left<u, \varphi\right>_{L^2_f\left(M\right)}.
        \end{equation}

        In general, we will write the following for our eigenspaces, where $A$ and $B$ are some function spaces defined on $M$:

        \begin{equation*}
            \mathcal{E}\left(\lambda; L_f, A, B\right) := \left\{u \in A : \mathcal{B}_f\left[u,\varphi\right] = \lambda \left<u,\varphi\right>_{L^2_f\left(M\right)},~\forall \varphi \in B\right\}.
        \end{equation*}

        If $A = B$, we will write

        \begin{equation*}
            \mathcal{E}\left(\lambda; L_f, A\right) := \mathcal{E}\left(\lambda; L_f, A, A\right). 
        \end{equation*}

        We can thus formulate our eigenspaces as

        \begin{equation*}
            \mathcal{E}\left(\lambda; L_f, W^{1,2}_f\left(M\right)\right) = \left\{u \in W^{1,2}_f\left(M\right) : \mathcal{B}_f\left[u,\varphi\right] = \lambda \left<u,\varphi\right>_{L^2_f\left(M\right)},~ \forall \varphi \in W^{1,2}_f\left(M\right)\right\}.
        \end{equation*}
        
        Now we can define the Morse index and nullity of a Ricci shrinker.

        \begin{definition}\label{index_nullity_defn}
            The \emph{f-index (Morse index)} and \emph{f-nullity} of a gradient Ricci shrinker are, respectively,

            \begin{align*}
                \mathrm{Ind}_f\left(M\right) &:= \sum_{\lambda < 0} \mathrm{dim}\left(\mathcal{E}\left(\lambda; L_f, W^{1,2}_f\left(M\right)\right)\right),\\
                \mathrm{Null}_f\left(M\right) &:= \mathrm{dim}\left(\mathcal{E}\left(0; L_f, W^{1,2}_f\left(M\right)\right)\right).
            \end{align*}
        \end{definition}

        As mentioned in Section \ref{introduction}, we implicitly do not consider $\mathrm{Ric}_g$ when computing $\mathrm{Ind}_f\left(M\right)$. We will often just refer to these as the index and nullity, as long as there is no chance of confusion. Also under the assumptions of Theorem \ref{spectral_theorem} we have $\mathrm{Ind}_f\left(M\right) + \mathrm{Null}_f\left(M\right) < \infty$. This is also true for the orbifold shrinker we get in the limit since, as outlined in Remark \ref{orbifold_remark}, the orbifolds we deal with are orbifolds of bounded curvature. Finally, one can show that

        \begin{equation*}
            \mathrm{Ind}_f\left(M\right) = \sup \left\{\mathrm{dim}\left(W\right) : W \subset L^2_f\left(M\right)~\mathrm{is~a~linear~space~such~that}~ \left.\mathcal{B}_f\left[u,u\right]\right|_W < 0\right\}
        \end{equation*}

        and analogously for the nullity.

    \subsection{The Einstein Operator on the Bubbles}
        Now for the bubbles. Note that after blowing up around an orbifold singularity we have $f_i \rightarrow C$ for some constant $C \in \mathbb{R}$, which is a consequence of the bubbles being Ricci-flat (that is, gradient steady solitons with constant potential function). Thus $L_{\widetilde{g}^k_i,f_i} \rightarrow L_{h^k}$. Here $\widetilde{g}^k_i := \left(s^k_i\right)^{-2} g_i$ and $L_{h^k}$ is the Einstein operator on the bubble $\left(V^k,h^k\right)$, which is precisely the usual stability operator encountered when studying Ricci-flat ALE manifolds.

        Since the bubbles are non-compact and a good notion of finite weighted volume is no longer available, we need to use other results/techniques to ensure that the spectrum behaves well enough for our purposes. In particular we want the index $\mathrm{Ind}\left(V\right) := \lim \limits_{R \rightarrow \infty} \mathrm{Ind}\left(B_h\left(q,R\right)\right)$ and $\mathrm{Null}\left(V\right) := \lim \limits_{R \rightarrow \infty} \mathrm{Null}\left(B_h\left(q,R\right)\right)$ to be finite. Here $q \in V$ is a point. We also emphasize that the index and nullity on the bubble count, respectively, the dimension of the negative and zero eigenspaces of $L_{h^k}$ on the bubble $\left(V^k,h^k\right)$.

        To ensure the desired finiteness of the index and nullity on the bubbles, we appeal to work by B\'erard--Besson in \cite{berard_besson} and Carron in \cite{carron}. Their proofs mainly rely on having a Euclidean type Sobolev inequality

        \begin{equation*}
            \left(\int_V \left|u\right|^{\frac{2n}{n-2}} dV_h\right)^{\frac{n-2}{n}} \leq C_S \int_V \left|\nabla u\right|^2 dV_h
        \end{equation*}

        for all $u \in C^\infty_c\left(V\right)$.

        \begin{theorem}\label{bubble_index_energy_bound}
            Let $\left(V,h\right)$ be an $n$-dimensional ALE bubble. Then $\mathrm{Ind}\left(V\right) + \mathrm{Null}\left(V\right) <  \infty$.
        \end{theorem}

        \begin{proof}
            We start by considering leaf bubbles. Since the bubble is Ricci-flat, we can use the Sobolev inequality from Theorem 3.3.8 in \cite{saloff_coste}. Then, as mentioned above, the proof follows along the lines of the ones in \cite{berard_besson,carron} used to derive so-called Cwickel--Lieb--Rosenbljum bounds. We omit these here for the sake of brevity. This all works on intermediate bubbles after passing to local orbifold covers.
        \end{proof}

        We will also eventually need to consider the nullity of $L$ acting as a map between polynomially weighted Sobolev spaces: 

        \begin{definition}\label{bubble_wgtd_sobolev}
            Let $\left(V,h\right)$ be an $n$-dimensional ALE bubble, $\beta \in \mathbb{R}$, and $\rho: V \rightarrow \left(0,\infty\right)$ a radius function. That is, along the ALE end of $V$ we have $\varphi_\ast \rho = r$, where $r$ is the radial distance function $\mathbb{R}^n$ and $\varphi$ is an ALE chart as in Definition \ref{ale_bubble}. Then the weighted Sobolev space $L^2_\beta\left(V\right)$ is the closure of $C^\infty_c\left(V\right)$ with respect to the norm

            \begin{equation*}
                \left|\left|u\right|\right|_{L^2_\beta\left(V\right)} := \left(\int_V \left|u\right|^2 \rho^{2\beta - n} dV_h\right)^{\frac{1}{2}}.
            \end{equation*}

            We also define $W^{k,2}_\beta\left(V\right)$ as the closure of $C^\infty_c\left(V\right)$ with respect to the norm

            \begin{equation*}
                \left|\left|u\right|\right|_{W^{k,2}_\beta\left(V\right)} := \sum^k_{\ell=0} \left|\left|\nabla^\ell u\right|\right|_{L^2_{\beta + \ell}\left(V\right)}.
            \end{equation*}
        \end{definition}

        Note that when $\beta = \frac{n}{2}$ we have $L^2_\beta\left(V\right) = L^2\left(V\right)$. Also, the spaces are independent of the radius function chosen. The importance of working with these is that, while $L$ is not a Fredholm operator as a map between the usual unweighted Sobolev spaces, in certain cases it will be a Fredholm map between the weighted spaces \textit{and} its weighted kernel will coincide with the unweighted kernel. This hinges on the weight parameter $\beta$ not being an ``exceptional value''. Roughly speaking, these are values of $\beta$ for which the kernel contains elements that look like $r^{-\beta}$ as $r \rightarrow \infty$. For more information we refer the reader to \cite{bartnik, der_ozuch_loj} and the references therein. We are lucky in that the $\beta$ we will eventually consider is non-exceptional and can thus apply the following result (Proposition $5.1$ from \cite{der_ozuch_loj}):

        \begin{proposition}\label{ale_fredholm}
            Let $\left(V,h\right)$ be an $n$-dimensional ALE bubble. Then, if $\beta \in \left(0,n-2\right) \cup \left(n-2,n\right)$, we have that

            \begin{equation*}
                L : W^{2,2}_\beta\left(V\right) \rightarrow L^2_{\beta + 2}\left(V\right)
            \end{equation*}

            is Fredholm. Moreover, if $\beta = \frac{n}{2} - 1$ and $n \geq 3$, then the kernel of $L$ between the weighted spaces is equal to the usual $L^2$-kernel $\ker_{L^2}\left(L\right)$.
        \end{proposition}

        In \cite{der_ozuch_loj}, this result is phrased for leaf bubbles, but it holds on intermediate bubbles after passing to orbifold covers. We also refer the reader to Remark $4.2$ and the associated discussions and results in \cite{kroncke_szabo} for how to adjust the weighted Sobolev spaces in Definition \ref{bubble_wgtd_sobolev} and the proof of Proposition \ref{ale_fredholm} more directly.

        We end this section by stating a technical lemma which tells us that all the (f-)index is captured by a compact set. The proof seems to be well-known by now and only relies on finiteness of the (f-)index, so we merely refer the reader to the proof of Lemma $2$ in \cite{tysk}.

        \begin{lemma}\label{finite_index_capture}
            Let $\left(V,h\right)$ be an $n$-dimensional ALE bubble. Then there is a compact set $K \subset V$ such that $\mathrm{Ind}\left(V\right) = \mathrm{Ind}\left(K\right)$. The analogous conclusion holds for manifolds $M$ with finite $f$-index: there is a compact subset $K' \subset M$ such that $\mathrm{Ind}_f\left(M\right) = \mathrm{Ind}_f\left(K'\right)$.
        \end{lemma}
\section{Upper Semi-Continuity}\label{upper_semi_cont}
    For the proof of the upper semi-continuity estimate, we need to perform a thorough analysis of how the eigenvalues/modes behave when a sequence of shrinkers bubble tree converges. To accomplish this, we will adapt recent work of Da Lio--Gianocca--Rivi\`ere on critical points of certain conformally invariant Lagrangians (\cite{dlgr_upper_semicontinuity}) and Workman for certain CMC hypersurfaces (\cite{workman_semicontinuity}). This latter paper showed that much of the analysis in the former can be simplified if adequate Sobolev inequalities are available.

    The general approach is the following. In order to ensure the eigenvalue problem scales appropriately when we blow-up around the orbifold points, we will work with a weighted version. However, the precise properties of the weight function will end up stopping us from proceeding by using the usual H\"older's inequality and a Sobolev inequality which holds on shrinkers at sufficiently small scales (see Lemma $3.2$ in \cite{HM1}). A similar discussion holds for the ALE bubbles. To bypass these issues, we will work with certain weak $L^p$-spaces (called ``Lorentz spaces''), associated to which are H\"older and Sobolev inequalities suited to our purpose. We then go through the same procedure as in Section \ref{bubble_tree_prelims}: first analyze the body region, then blow up to study the bubble regions, and finally show no concentration occurs in the necks. The final ingredient is showing the weighted and unweighted eigenspaces have the same dimension on the body and bubbles.

    Moving forward we will often abuse notation slightly and, for $i \gg 1$, write $q_i \in \mathcal{Q}$ to mean a sequence of points $q_i \in M_i$ which converge to an orbifold point $q \in \mathcal{Q}$. Analogous statements hold when referring to regions converging to a subset of the body, neck, or leaf/intermediate bubble regions.

    \subsection{Lorentz Spaces}\label{lorentz_spaces}
        We now define Lorentz spaces and state the associated H\"older and Sobolev inequalities. The point of the following definition is that $d^{-2}_g\left(\cdot,q\right) \in L^{\frac{n}{2},\infty}\left(M\right)$, provided the space has Euclidean volume growth (see Lemma \ref{weight_lorentz_bound} and the remark following the proof). Roughly speaking, this is because $d^{-2}_g\left(\cdot,q\right)$ becoming arbitrarily large is balanced out by the volume of the region where such behavior occurs. A similar discussion holds when considering $\left(s^k_i\right)^{-2}$ with $s^k_i$ a bubble scale.
    
        \begin{definition}\label{lorentz_defn}
            Let $\left(\Omega,\mu\right)$ be a $\sigma$-finite non-atomic measure space. Further consider $u \in \Gamma\left(\Omega, E\right)$ which is measurable with respect to $\mu$ and define the following:
    
            \begin{equation*}
                \alpha_{u,\Omega}\left(s\right) := \mu\left(\left\{x \in \Omega : \left|u\right|\left(x\right) > s\right\}\right).
            \end{equation*}
    
            The decreasing rearrangement of $u$ on $\Omega$, denoted $u^\ast_\Omega$, is then defined as
    
            \begin{equation*}
                u^\ast_\Omega\left(t\right) := \begin{cases}
                    \inf \left\{s > 0 : \alpha_{u,\Omega}\left(s\right) \leq t\right\},~~ &t > 0\\
                    \mathrm{ess\,sup}_\Omega \left|u\right|,~~ &t = 0.
                \end{cases}
            \end{equation*}
    
            For $p \in \left[1,\infty\right)$, $q \in \left[1,\infty\right]$, and $u$ as above, define the following quasi-norm:
    
            \begin{equation*}
                \left|\left|u\right|\right|_{L^{p,q}_\mu\left(\Omega\right)} := \begin{cases}
                    \left(\int^\infty_0 t^{\frac{q}{p}} \left(u^\ast_\Omega\left(t\right)\right)^q \frac{dt}{t}\right)^{\frac{1}{q}},~~ &1 \leq q < \infty\\
                    \sup_{t > 0} t^{\frac{1}{p}} u^\ast_\Omega\left(t\right),~~ &q = \infty.
                \end{cases}
            \end{equation*}
    
            The \textit{Lorentz space} $L^{p,q}_\mu\left(\Omega\right)$ with respect to the measure $\mu$ is then defined as the space of all $u \in \Gamma\left(\Omega,E\right)$ such that $\left|\left|u\right|\right|_{L^{p,q}_\mu\left(\Omega\right)} < \infty$.
        \end{definition}
    
        If $p=q$ then one can show that $L^{p,p}\left(M\right) = L^p\left(M\right)$. In practice, we will either take $d\mu = dV_g$  or $d\mu = e^{-f} dV_g$. We will denote the associated spaces as, respectively, $L^{p,q}$ and $L^{p,q}_f$.  Also, note that the Lorentz spaces we just defined are not Banach spaces. However, one can define a norm, say $\left|\left|\cdot\right|\right|_{L^{\left(p,q\right)}_\mu\left(\Omega\right)}$, which is comparable to $\left|\left|\cdot\right|\right|_{L^{p,q}_\mu\left(\Omega\right)}$. One can then show the Lorentz spaces are complete with respect to $\left|\left|\cdot\right|\right|_{L^{\left(p,q\right)}_\mu\left(\Omega\right)}$. We refer the reader to Chapter $2$ of \cite{castillo_chaparro} for more details about the completeness, norms, and also separability of Lorentz spaces.
    
        The Lorentz space variant of H\"older's inequality is the following (Theorem $2.9$ in \cite{castillo_chaparro}):
        
        \begin{proposition}\label{lorentz_holder}
           Let $\left(\Omega, \mu\right)$ be a $\sigma$-finite non-atomic measure space. Then, for $p_1,p_2 \in \left(1,\infty\right)$ and $q_1,q_2 \in \left[1,\infty\right]$ such that $\frac{1}{p_1} + \frac{1}{p_2} = p$ and $\frac{1}{q_1} + \frac{1}{q_2} = q$, we have 
    
           \begin{equation*}
               \left|\left|\left<u_1,u_2\right>\right|\right|_{L^{p,q}_\mu\left(\Omega\right)} \leq \left|\left|u_1\right|\right|_{L^{p_1,q_1}_\mu\left(\Omega\right)}\left|\left|u_2\right|\right|_{L^{p_2,q_2}_\mu\left(\Omega\right)}
           \end{equation*}
    
           provided $u_1 \in L^{p_1,q_1}_\mu\left(\Omega\right)$ and $u_2 \in L^{p_2,q_2}_\mu\left(\Omega\right)$.
        \end{proposition}
    
        We now state the Lorentz--Sobolev inequality and prove uniform local control on the constant in the estimate, provided we work in a small enough region.
    
        \begin{proposition}\label{lorentz_sobolev}
            Let $\left(M,g\right)$ be an $n$-dimensional complete connected (orbifold) gradient Ricci shrinker with $\mu\left(g\right) \geq \underline{\mu} > -\infty$ and let $r > 0$. Then there are constants $C^1_{\mathrm{LS}}, C^2_{\mathrm{LS}}, \delta_0 > 0$ depending only on $n, \underline{\mu}, r$ such that, for every ball $B_g\left(x, \delta\right) \subset B_g\left(p,r\right)$ with $\delta \leq \delta_0$, we have
    
            \begin{equation}\label{lorentz_sob_cc}
                \left|\left|u\right|\right|_{L^{\frac{2n}{n-2},2}_f\left(B_g\left(x,\delta\right)\right)} \leq C^1_{\mathrm{LS}}\left|\left|\nabla u\right|\right|_{L^2_f\left(B_g\left(x,\delta\right)\right)}
            \end{equation}
    
            for any $u \in C^\infty_c\left(B_g\left(x,\delta\right)\right)$ and
    
            \begin{equation}\label{lorentz_sob_extended}
                \left|\left|u\right|\right|_{L^{\frac{2n}{n-2},2}_f\left(B_g\left(x,\delta\right)\right)} \leq C^2_{\mathrm{LS}}\left|\left|u\right|\right|_{W^{1,2}_f\left(B_g\left(x,\delta\right)\right)}
            \end{equation}
    
            for any $u \in W^{1,2}_f\left(B_g\left(x,\delta\right)\right)$.
        \end{proposition}
    
        \begin{proof}
            As a byproduct of the proof of Lemma $3.2$ in \cite{HM1}, there is a constant $C > 0$ depending only on $n, \underline{\mu},r$ such that, for a sufficiently small $\delta_0$ as in the statement of the proposition and any $\delta \leq \delta_0$,
    
            \begin{equation*}
                \left|B_g\left(p,\delta\right)\right|^{\frac{n-1}{n}} \geq C\left|\partial B_g\left(p,\delta\right)\right|.
            \end{equation*}
    
            That is, a Euclidean type isoperimetric inequality holds on sufficiently small regions with uniform local control on the constant $C$. With this in hand, we can appeal to results of Maz'ya contained in Section $2.3.1$ and Corollary $2.2.3/2$ in \cite{mazya} to deduce that, for a positive constant $C'$ depending only on $n, \underline{\mu}, r$, we have
    
            \begin{equation*}
                \left|\left|u\right|\right|_{L^{\frac{2n}{n-2},2}\left(B_g\left(p,\delta\right)\right)} \leq C'\left|\left|\nabla u\right|\right|_{L^2\left(B_g\left(p,\delta\right)\right)}
            \end{equation*}
    
            for all $u \in C^\infty_c\left(B_g\left(x,\delta\right)\right)$. Therefore, since \eqref{potential_growth} tells us $dV_g$ and $e^{-f}dV_g$ are comparable on compact subsets of $M$,
    
            \begin{equation*}
                \left|\left|u\right|\right|_{L^{\frac{2n}{n-2},2}_f\left(B_g\left(p,\delta\right)\right)} \leq C''\left|\left|\nabla u\right|\right|_{L^2_f\left(B_g\left(p,\delta\right)\right)}
            \end{equation*}
    
            for all $u \in C^\infty_c\left(B_g\left(x,\delta\right)\right)$ and some positive constant $C''$ depending only on $n, \underline{\mu}, r$. This yields \eqref{lorentz_sob_cc}. A density argument and the discussion about the separability and completeness of Lorentz spaces from earlier yields \eqref{lorentz_sob_extended} and completes the proof.
        \end{proof}
    
        Note that in our setting the radius $r$ in Proposition \ref{lorentz_sobolev} can be taken to be $r_{\mathrm{orb}}$ and thus controlled in terms of $n, \underline{\mu}, R_0$. Next, using both Proposition \ref{lorentz_holder} and Proposition \ref{lorentz_sobolev}, we can prove the following.
    
        \begin{proposition}\label{triple_prod_estim}
            Let $\left(M,g\right)$ be an $n$-dimensional complete connected (orbifold) gradient Ricci shrinker with $\mu\left(g\right) \geq \underline{\mu} > -\infty$. Then, for $u,v \in W^{1,2}_f\left(B_g\left(x, \delta\right)\right)$, $B_g\left(x, \delta\right) \subset B_g\left(p, r\right)$ as in Proposition \ref{lorentz_sobolev}, and $\omega \in L^{\frac{n}{2},\infty}_f\left(B_g\left(x, \delta\right)\right)$, we have
    
            \begin{equation*}
                \left|\int_{B_g\left(x, \delta\right)} \omega \left<u,v\right> e^{-f}dV_g\right| \leq C\left|\left|\omega\right|\right|_{L^{\frac{n}{2},\infty}_f\left(B_g\left(x, \delta\right)\right)}\left|\left|u\right|\right|_{W^{1,2}_f\left(B_g\left(x, \delta\right)\right)}\left|\left|v\right|\right|_{W^{1,2}_f\left(B_g\left(x, \delta\right)\right)}.
            \end{equation*}
    
            Here $C > 0$ is a constant depending only on $n, \underline{\mu},r$.
        \end{proposition}
    
        \begin{proof}
            To ease notation, set $\Omega := B_g\left(x, \delta\right)$. Using Proposition \ref{lorentz_holder} with $p=1$, $p_1 = \frac{n}{2}$, $p_2 = \frac{n}{n-2}$ and $q = 1$, $q_1 = \infty$, $q_2 = 1$ we have
    
            \begin{equation*}
                \left|\int_\Omega \omega \left<u,v\right> e^{-f}dV_g\right| \leq \left|\left|\omega\right|\right|_{L^{\frac{n}{2},\infty}_f\left(\Omega\right)} \left|\left|\left<u,v\right>\right|\right|_{L^{\frac{n}{n-2},1}_f\left(\Omega\right)}.
            \end{equation*}
    
            Applying Proposition \ref{lorentz_holder} again, this time with $p = \frac{n}{n-2}$, $p_1 = \frac{2n}{n-2}$, $p_2 = \frac{2n}{n-2}$, $q = 1$, $q_1 = 2$, $q_2 = 2$, and then \eqref{lorentz_sob_extended}, we get
    
            \begin{align*}
                \left|\left|\left<u,v\right>\right|\right|_{L^{\frac{n}{n-2},1}_f\left(\Omega\right)} &\leq \left|\left|u\right|\right|_{L^{\frac{2n}{n-2},2}_f\left(\Omega\right)}\left|\left|v\right|\right|_{L^{\frac{2n}{n-2},2}_f\left(\Omega\right)}\\
                &\leq C^2_{LS}\left|\left|u\right|\right|_{W^{1,2}_f\left(\Omega\right)}\left|\left|v\right|\right|_{W^{1,2}_f\left(\Omega\right)}.
            \end{align*}
    
            Combining this with the previous step gives what we wanted.
        \end{proof}
    
        Moving forward, we will typically apply Proposition \ref{triple_prod_estim} to a bubble region associated to the point scale sequence $\left(q^k_i, s^k_i\right)$, say $\Omega = B_{g_i}\left(q^k_i, Rs^k_i\right)$ for some $R > 0$, or certain subsets of such a ball. Note that this is always possible provided we take $i \gg 1$ so the bubble scales act as sufficiently small radii.
    
        \begin{remark}\label{triple_prod_remark}
            An inspection of the proof of Propositon \ref{triple_prod_estim} tells us one can prove
    
            \begin{equation*}
                \left|\int_{B_g\left(x, \delta\right)} \omega \left<u,v\right> e^{-f}dV_g\right| \leq C\left|\left|\omega\right|\right|_{L^{\frac{n}{2},\infty}_f\left(B_g\left(x, \delta\right)\right)}\left|\left|\nabla u\right|\right|_{L^2_f\left(B_g\left(x, \delta\right)\right)}\left|\left|\nabla v\right|\right|_{L^2_f\left(B_g\left(x, \delta\right)\right)}
            \end{equation*}
    
            for $u, v \in C^\infty_c\left(B_g\left(x, \delta\right)\right)$ by using \eqref{lorentz_sob_cc} in place of \eqref{lorentz_sob_extended}. Here $C > 0$ is a constant depending only on $n, \underline{\mu}, r$.
        \end{remark}
    \subsection{The Weighted Eigenvalue Problem}\label{weighted_problem}
        Before starting the proof of the upper semi-continuity result, we need to define an appropriate weighted eigenvalue problem. This is needed as eigenvalues may concentrate around an orbifold point and we will blow-up to analyze the corresponding eigenvalue problem on the bubble. However, to make sure this is meaningful, we need to modify \eqref{eigenvalue_weak} so the left and right hand sides of it have the same scaling, hence the introduction of a weight. To this end, let $S > 0$, $i \gg 1$, and $0 < \delta \ll 1$ be such that, for each $k = 1,\dots,T$, we have $S s^k_i < \delta$, where $\left(q^k_i,s^k_i\right)$ is a point-scale sequence as in Section \ref{bubble_tree_prelims}. Then define
    
        \begin{equation*}
            \omega_{i,S,\delta,k}\left(x\right) := \begin{cases}
                \max \left\{\delta^{-2}, d^{-2}_{g_i}\left(x, q^k_i\right)\right\},~~&x \in M_i \backslash B_{g_i}\left(q^k_i, Ss^k_i\right)\\
                \left(S s^k_i\right)^{-2},~~&x \in B_{g_i}\left(q^k_i,Ss^k_i\right).
            \end{cases}
        \end{equation*}
    
        Our \emph{weight function} is then
    
        \begin{equation*}
            \omega_{i,S,\delta}\left(x\right) := \max_{k = 1, \dots, T}\omega_{i,S,\delta,k}\left(x\right),
        \end{equation*}
    
        where $T$ is the number of bubbles that form. One can deduce that, on $M_\infty \backslash \mathcal{Q}$, we have
    
        \begin{equation}\label{limiting_body_weight}
            \omega_{i,S,\delta}\left(x\right) \rightarrow \omega_{\infty,\delta}\left(x\right) := \max \left\{\delta^{-2}, d^{-2}_{g_\infty}\left(x,\mathcal{Q}\right)\right\}
        \end{equation}
    
        as $i \rightarrow \infty$. Here the convergence is in the $W^{1,\infty}_{\mathrm{loc}}\left(M_\infty \backslash \mathcal{Q}\right)$-sense. Also, for this section and the rest of the paper, we set $K_i := B_{g_i}\left(p_i, r_{\mathrm{orb}}\right)$, with $r_{\mathrm{orb}}$ as defined at the end of Section \ref{preliminaries}. In particular, we have $\mathcal{Q} \subset K_\infty$, thus $\sup_{M_\infty \backslash K_\infty} \left|\mathrm{Rm}_{g_\infty}\right| < \infty$ and likewise on $M_i \backslash K_i$. We now show that $\omega_{i,S,\delta} \in L^{\frac{n}{2},\infty}_{f_i}\left(K_i\right)$ and $\omega_{\infty,\delta} \in L^{\frac{n}{2},\infty}_{f_\infty}\left(K_\infty\right)$. This is due to the following lemma and its proof, which is almost exactly as in Section $2.1$ of \cite{workman_semicontinuity}, so we only mention the needed modifications.
    
        \begin{lemma}\label{weight_lorentz_bound}
            For $n \geq 4$, assume $\left(M_i, g_i, f_i\right)$ is a sequence of $n$-dimensional complete connected gradient Ricci shrinkers satisfying $\left(\mathcal{A}\right)$. Then
    
            \begin{equation*}
                \left|\left|\omega_{i,S,\delta}\right|\right|_{L^{\frac{n}{2},\infty}_{f_i}\left(K_i\right)} \leq C\left(R_0,n,\underline{\mu},\delta,S,T\right),
            \end{equation*}
    
            where $\delta \in \left[0,\min\left\{1, \frac{\mathrm{inj}\left(M_i\right)}{2}\right\}\right)$.
        \end{lemma}
    
        \begin{proof}
            First consider $u_k := \omega_{i,S,\delta,k}$. One can compute $\alpha_{u_k, K_i}\left(s\right)$ and $u^\ast_{k,K_i}\left(t\right)$ directly just like in Section $2$ of \cite{workman_semicontinuity}. These quantities can then be bounded using the comparability of the measures $e^{-f_i}dV_{g_i}$ and $dV_{g_i}$ on $K_i$ followed by the Euclidean volume growth of $K_i$, which is due to \eqref{grs_upper_volume}. Summing over $k$ and and using $\sum_{q \in \mathcal{Q}} N_q = T < \infty$ completes the proof.
        \end{proof}
    
        \begin{remark}\label{distance_lorentz}
            The proof of Lemma \ref{weight_lorentz_bound} can also be adapted to show $\omega_{\infty,\delta} \in L^{\frac{n}{2},\infty}_{f_\infty}\left(K_\infty\right)$. Note that this really says that $d^{-2}_g\left(\cdot,\mathcal{Q}\right) \in L^{\frac{n}{2},\infty}_{f_\infty}\left(K_\infty\right)$ and carries over to our ALE bubbles. Alternatively, one can use Lemma \ref{weight_lorentz_bound} and appeal to the lower semi-continuity of norms.
        \end{remark}
    
        We can also relate the Riemann curvature and the weight function.
    
        \begin{lemma}\label{rm_wgt_ineq}
            For $n \geq 4$, assume $\left(M_i, g_i, f_i\right)$ is a sequence of $n$-dimensional complete connected gradient Ricci shrinkers satisfying $\left(\mathcal{A}\right)$. Then, for $i \gg 1$ and $\delta \ll 1$, we have
    
            \begin{equation*}
                \left|\mathrm{Rm}_{g_i}\right|\left(x\right) \leq C\omega_{i,S,\delta}\left(x\right)
            \end{equation*} 
    
            for all $x \in M_i$, where $C = C\left(R_0,S,\delta\right) > 0$. The same result holds for the limiting weight on the orbifold shrinker limit $\left(M_\infty, g_\infty, f_\infty\right)$:
    
            \begin{equation*}
                \left|\mathrm{Rm}_{g_\infty}\right|\left(x\right) \leq C'\omega_{\infty,S,\delta}\left(x\right)
            \end{equation*}
    
            for all $x \in M_\infty$, where $C' = C'\left(R_0,S,\delta, \mathcal{Q}\right) > 0$.
        \end{lemma}
    
        \begin{proof}
            Assume for a contradiction that the desired result does not hold. Then there is a sequence of points $x_i \in M_i$ such that, as $i \rightarrow \infty$, 
    
            \begin{equation}\label{rm_wgt_contra}
                \frac{\left|\mathrm{Rm}_{g_i}\right|\left(x_i\right)}{\omega_{i,S,\delta}\left(x_i\right)} \rightarrow \infty.
            \end{equation}
    
            If $x_\infty \in M_\infty \backslash \mathcal{Q}$ then for $i \gg 1$ we would have $\omega_{i,S,\delta}\left(x_i\right) \geq \max\left\{\left(2\delta\right)^{-2},\left(2d_{g_\infty}\left(x_\infty,\mathcal{Q}\right)\right)^{-2}\right\}$, but in such a case \eqref{rm_wgt_contra} contradicts the body region having bounded curvature. 
    
            Next, we consider when $x_i$ accumulates in a bubble region. Consider first the case of a leaf bubble region, say $B_{g_i}\left(q^1_i, Rs^1_i\right)$. Then, since $\sup_{B_{g_i}\left(q^1_i, Rs^1_i\right)} \left|\mathrm{Rm}_{g_i}\right|\left(x_i\right) = \left(s^1_i\right)^{-2}$,
    
            \begin{equation*}
                \frac{\left|\mathrm{Rm}_{g_i}\right|\left(x_i\right)}{\omega_{i,S,\delta}\left(x_i\right)} \leq 2\max\left\{S,R\right\}^2 \left(s^1_i\right)^2 \sup \limits_{B_{g_i}\left(q^1_i, Rs^1_i\right)} \left|\mathrm{Rm}_{g_i}\right|\left(x_i\right) = 2\max\left\{S,R\right\}^2.
            \end{equation*}
    
            \sloppy This then contradicts \eqref{rm_wgt_contra} and the argument can be repeated for each leaf bubble region. Now assume $x_i$ lies in an intermediate bubble region with child bubbles excised, say $B_{g_i}\left(q^2_i,Rs^2_i\right) \backslash B_{g_i}\left(q^1_i, \frac{1}{R}s^1_i\right)$ for simplicity. Then we can apply the previous argument to get a contradiction and proceed similarly on every other intermediate bubble region after excising child bubbles appropriately.
    
            Finally, we consider the case when $x_i$ accumulates in a neck region $\mathcal{N}_{i,R}\left(q_i\right)$. Note then that this means for $i \gg 1$ we have $x_i \in B_{g_i}\left(p_i, r_{\mathrm{orb}} + 2\right)$. The precise bubble scales involved in the neck region here do not influence the proof, so we suppress them for simplicity. We now need to handle $2$ subcases. First, when $x_i$ is such that $\omega_{i,S,\delta}\left(x_i\right) = d^{-2}_{g_i}\left(x_i,q_i\right)$ for $i \gg 1$. In such a situation, we have
    
            \begin{equation*}
                \frac{\left|\mathrm{Rm}_{g_i}\right|\left(x_i\right)}{\omega_{i,S,\delta}\left(x_i\right)} = d^2_{g_i}\left(x_i,q_i\right) \left|\mathrm{Rm}_{g_i}\right|\left(x_i\right).
            \end{equation*}
    
            Also, by Claim $5.5$ in \cite{by_bubbling} we know that the energy of the neck region $\mathcal{N}_{i,R}\left(q_i\right)$ must go to $0$ as $i, R \rightarrow \infty$. Moreover, we note $r_{\mathrm{orb}}$ uniformly depends on only $R_0, n, \underline{\mu}$. We can thus apply Lemma \ref{hm_eps_reg} with a uniform choice for $r$. As discussed earlier, this means the constants in the statement of Lemma \ref{hm_eps_reg} are uniform. Therefore, after possibly passing to a subsequence, we get 
    
            \begin{equation*}
                d^2_{g_i}\left(x_i,q_i\right) \left|\mathrm{Rm}_{g_i}\right|\left(x_i\right) \leq K \left|\left|\mathrm{Rm}_{g_i}\right|\right|_{L^{\frac{n}{2}}\left(\mathcal{N}_i\left(q_i\right)\right)} \rightarrow 0
            \end{equation*}
    
            as $i, R \rightarrow \infty$, which contradicts \eqref{rm_wgt_contra}. The second subcase is when $\omega_{i,S,\delta}\left(x_i\right) = \left(Ss^k_i\right)^{-2}$. We can proceed by using Claim $5.7$ from \cite{by_bubbling} to control the bubble scale in terms of the outer radius of the neck region and then appealing to Lemma \ref{hm_eps_reg} again to derive a contradiction.
    
            As for the limiting orbifold, since these are singularities of conical type we have $\left|\mathrm{Rm}_{g_\infty}\right|\left(x\right) \leq C_q d^{-2}_{g_\infty}\left(x,q\right)$ in neighborhoods of each $q \in \mathcal{Q}$. Combining this with \eqref{limiting_body_weight} and the logic of the previous arguments we have the desired result. 
        \end{proof}
    
        \begin{remark}\label{bubble_adjustment}
            An inspection of the proofs of Proposition \ref{triple_prod_estim}, Lemma \ref{weight_lorentz_bound}, and Lemma \ref{rm_wgt_ineq} shows they hold on (compact subsets of) ALE bubbles after replacing the $f$-weighted Sobolev spaces with the usual ones. We also note that Lemma \ref{rm_wgt_ineq} holds on cones or manifolds with general conical singularities since it relies on the conical, rather than orbifold, structure of the singular set.
        \end{remark}
    
        Using the weight function $\omega_{i,S,\delta}$, we define the following weighted eigenvalue problem:
    
        \begin{equation}\label{weighted_eigenvalue_weak}
            \mathcal{B}_f\left[u_i,\varphi\right] = \lambda_i \int_M \omega_{i,S,\delta} \left<u_i, \varphi\right>~e^{-f_i} dV_{g_i}
        \end{equation}
    
        for $u_i, \varphi \in W^{1,2}_{f_i}\left(M_i\right)$. We will often write the right hand side of \eqref{weighted_eigenvalue_weak} as a weighted $L^2_{f_i}$-inner product:
    
        \begin{equation}\label{weighted_inner_product}
            \left<u_i,\varphi\right>_{\omega_{i,S,\delta},L^2_{f_i}\left(M_i\right)} := \int_{M_i} \omega_{i,S,\delta} \left<u_i,\varphi\right>~e^{-f_i} dV_{g_i}.
        \end{equation}
    
        Note that this inner product is well-defined by Proposition \ref{triple_prod_estim} and $\omega_{i,S,\delta} \in L^\infty\left(M_i \backslash K_i\right)$. Versions of the weighted problem and inner product for the orbifold limit and the bubbles are defined analogously using \eqref{limiting_body_weight}.
    
        We define the weighted eigenspaces as
    
        \begin{equation*}
            \mathcal{E}_{\omega_{i,S,\delta}}\left(\lambda_i; L_{f_i}, A, B\right) := \left\{u \in A : \mathcal{B}_f\left[u,\varphi\right] = \lambda_i \left<u_i,\varphi\right>_{\omega_{i,S,\delta},L^2_{f_i}\left(M_i\right)}~~\mathrm{for~all}~ \varphi \in B\right\}
        \end{equation*}
    
        where $A$ and $B$ are function spaces. When $A = B$ we set
    
        \begin{equation*}
            \mathcal{E}_{\omega_{i,S,\delta}}\left(\lambda_i; L_{f_i}, A\right) := \mathcal{E}_{\omega_{i,S,\delta}}\left(\lambda_i; L_{f_i}, A, A\right).
        \end{equation*}
    
        The weighted f-index and weighted f-nullity are defined as
    
        \begin{align*}
            \mathrm{Ind}_{f_i,\omega_{i,S,\delta}}\left(M_i\right) &:= \sum_{\lambda_i < 0} \mathrm{dim}\left(\mathcal{E}_{\omega_{i,S,\delta}}\left(\lambda_i; L_{f_i}, W^{1,2}_{f_i}\left(M_i\right)\right)\right),\\
            \mathrm{Null}_{f_i,\omega_{i,S,\delta}}\left(M_i\right) &:= \mathrm{dim}\left(\mathcal{E}_{\omega_{i,S,\delta}}\left(0; L_{f_i}, W^{1,2}_{f_i}\left(M_i\right)\right)\right).
        \end{align*}
    
        The weighted eigenspaces on the limiting orbifold shrinker $\left(M_\infty,g_\infty,f_\infty\right)$ with stability operator $L_{f_\infty}$, and the corresponding weighted f-index and weighted f-nullity, are all defined analogously.
    \subsection{Body Region Analysis}\label{body_upper}
        We now turn to analyzing the convergence of our eigenvalues/modes on the body region $M_\infty \backslash \mathcal{Q}$. Consider a sequence of symmetric $2$-tensors $u_i \in W^{1,2}_{f_i}\left(M_i\right)$, which solve \eqref{weighted_eigenvalue_weak} with $\lambda_i \leq 0$. We also normalize these eigenmodes so that, for each $i$,
        
        \begin{equation}\label{body_eigenmode_normal}
            \int_{M_i} \omega_{i,S,\delta} \left|u_i\right|^2~e^{-f_i} dV_{g_i} = 1.
        \end{equation}
        
        Using Lemma \ref{rm_wgt_ineq}, \eqref{weighted_eigenvalue_weak}, and \eqref{body_eigenmode_normal}, we can guarantee the existence of a positive constant $C$ such that $\lambda_i \geq -C\left(R_0,S,\delta\right)$. Next, using Lemma \ref{rm_wgt_ineq}, \eqref{weighted_eigenvalue_weak}, and \eqref{body_eigenmode_normal} again, as well as $\lambda_i \leq 0$, yields
    
        \begin{align*}
            \int_{M_i} \left|\nabla u_i\right|^2 e^{-f_i}dV_{g_i} &= \lambda_i \int_{M_i} \omega_{i,S,\delta} \left|u_i\right|^2~e^{-f_i}dV_{g_i} + \int_{M_i} 2\left<\mathrm{Rm}_{g_i} \ast u_i, u_i\right> e^{-f_i} dV_{g_i}\\
            &\leq C\left(R_0,S,\delta\right) \int_{M_i} \omega_{i,S,\delta} \left|u_i\right|^2 e^{-f_i} dV_{g_i}\\
            &\leq C\left(R_0,S,\delta\right).
        \end{align*}
    
        Furthermore, by the definition of the weight function $\omega_{i,S,\delta}\left(x\right) \geq \delta^{-2}$ for all $x \in M_i$. We therefore have
    
        \begin{equation*}
            \delta^{-2}\left|\left|u_i\right|\right|^2_{L^2_{f_i}\left(M_i\right)} \leq \int_{M_i} \omega_{i,S,\delta} \left|u_i\right|^2~e^{-f_i} dV_{g_i} \leq 1.
        \end{equation*}
    
        Rearranging this yields
        
        \begin{equation*}
            \left|\left|u_i\right|\right|_{L^2_{f_i}\left(M_i\right)} \leq \delta.
        \end{equation*}
    
        Putting everything together we see that, possibly after passing to a subsequence, we can assume $\lambda_i \rightarrow \lambda_\infty \leq 0$ and we have
    
        \begin{equation*}
            \left|\left|u_i\right|\right|_{W^{1,2}_{f_i}\left(M_i\right)} \leq C\left(R_0, S, \delta\right).
        \end{equation*}
    
        Moreover, away from bubble regions, by Proposition \ref{weighted_sobolev_compact} we can guarantee the existence of a symmetric $2$-tensor $u_\infty \in W^{1,2}_{f_\infty}\left(M_\infty \backslash \mathcal{Q}\right)$ such that
    
        \begin{equation*}
            \begin{cases}
                u_i \rightharpoonup u_\infty~~\mathrm{in}~W^{1,2}_{f_\infty}\left(M_\infty \backslash \mathcal{Q}\right)\\
                u_i \rightarrow u_\infty~~\mathrm{in}~L^2_{f_\infty}\left(M_\infty \backslash \mathcal{Q}\right).
            \end{cases}
        \end{equation*}
    
        This limiting eigenmode $u_\infty$ can be extended over the singular set $\mathcal{Q}$ in two ways. First, one can appeal to the uniform bounds we derived above and the lower semi-continuity of norms under weak convergence. Alternatively, we can use a $2$-capacity argument since $\mathcal{H}^{n-2}\left(\mathcal{Q}\right) = 0$. The argument seems to be standard, albeit with a couple different variations. Here, for instance, one can refer to the proof of Proposition $3.14$ in \cite{dai_honda_pan_wei}. Note that this also yields $\left|\left|u_\infty\right|\right|_{W^{1,2}_{f_\infty}\left(M_\infty\right)} \leq C\left(R_0,S,\delta\right)$ and tells us $W^{1,2}_{f_\infty}\left(M_\infty\right)$ coincides with the $W^{1,2}_{f_\infty}$-closure of $C^\infty_c\left(M_\infty \backslash \mathcal{Q}\right)$.
    
        Now we need to show that \eqref{weighted_eigenvalue_weak} holds on $M_\infty$. First consider \eqref{weighted_eigenvalue_weak} on $K_\infty$. Let $\Omega \subset \subset K_\infty \backslash \mathcal{Q}$. Since $\omega_{i,S,\delta} \rightarrow \omega_{\infty,\delta}$ in $W^{1,\infty}\left(\Omega\right)$, we see that \eqref{weighted_eigenvalue_weak} holds for all $\varphi \in C^\infty_c\left(\Omega\right)$. Standard elliptic regularity theory for linear systems (for instance, a slight adjustment to the proof Theorem $4.9$ in \cite{giaquinta} using $e^{-f} dV_g$ is equivalent to $dV_g$ on compact sets) and the fundamental theorem of the calculus of variations tells us $u_\infty \in W^{2,2}_{f_\infty}\left(\Omega\right)$ and $L_{f_\infty} u_\infty + \lambda_\infty u_\infty \omega_{\infty,\delta} = 0$ almost everywhere on $\Omega$. This implies \eqref{weighted_eigenvalue_weak} holds for all $\varphi \in C^\infty_c\left(K_\infty \backslash \mathcal{Q}\right)$. 
    
        To handle $M_\infty \backslash K_\infty$, we can proceed similarly to before and note that, at first, $u_\infty \in W^{2,2}_{f_\infty, \mathrm{loc}}\left(M_\infty \backslash K_\infty\right)$. We next turn this in to a uniform bound on all of $M_\infty \backslash K_\infty$. To do this, we consider the cut-off function $\psi^r$ which is defined as follows for $r > 2r_{\mathrm{orb}}$:
    
        \begin{equation*}
            \psi^r\left(x\right) := 
            \begin{cases}
                \chi\left(x\right),&~~\mathrm{on}~ A^{g_\infty}_{r_\mathrm{orb},2r_{\mathrm{orb}}}\left(p_\infty\right)\\
                1,&~~\mathrm{on}~ A^{g_\infty}_{2r_{\mathrm{orb}},r}\left(p_\infty\right)\\
                \varphi^r\left(x\right),&~~\mathrm{on}~ A^{g_\infty}_{r,2r}\left(p_\infty\right).
            \end{cases}
        \end{equation*} 
    
        Here $\chi$ is a smooth function vanishing on $K_\infty$ and identically $1$ on $A^{g_\infty}_{2r_{\mathrm{orb}},r}\left(p_\infty\right)$ such that $\left|\nabla \chi\right| + \left|\Delta \chi\right| \leq C$ for a universal constant $C > 0$, while $\varphi^r$ is the cut-off function from Section $3$ of \cite{li_zhang_weighted_estimates}. This latter function is such that $\left|\nabla \varphi^r\right| \leq C'r^{-1}$ and $\left|\Delta_f \varphi^r\right| \leq C'$ on $M_\infty \backslash K_\infty$ for a universal constant $C' > 0$. Since $u_\infty \in L^2_{f_\infty}\left(M_\infty\right)$ and $\mathrm{Rm}_{g_\infty}, \omega_{\infty,\delta} \in L^\infty\left(M_\infty \backslash K_\infty\right)$ we can consider $\widetilde{u}_\infty := \psi^r u_\infty \in L^2_{f_\infty}\left(M_\infty\right)$. Then we can compute $\Delta_f \widetilde{u}_\infty$ and use the triangle inequality to deduce that, after sending $r \rightarrow \infty$,
    
        \begin{align*}
            \left|\left|\Delta_f \widetilde{u}_\infty\right|\right|_{L^2_{f_\infty}\left(M_\infty\right)} &\leq C\left(n,R_0,\underline{\mu}\right)\left[\left(1+ \left|\lambda_\infty\right| \sup_{M_\infty \backslash K_\infty} \omega_{\infty,\delta}\right)\left|\left|u_\infty\right|\right|_{L^2_{f_\infty}\left(M_\infty \backslash K_\infty\right)} + \left|\left|\nabla u_\infty\right|\right|_{L^2_{f_\infty}\left(M_\infty \backslash K_\infty\right)}\right]\\
            &\leq C\left(n,R_0,\underline{\mu},\delta\right)\left|\left|u_\infty\right|\right|_{W^{1,2}_{f_\infty}\left(M_\infty\right)}
        \end{align*}
    
        Note we have also used the bound $\left|\lambda_\infty\right| \leq C\left(R_0,S,\delta\right)$. Therefore $\Delta_f \widetilde{u}_\infty \in L^2_{f_\infty}\left(M_\infty\right)$ so we can apply elliptic regularity estimates adapted to the drift Laplacian $\Delta_f$, in particular Proposition $3.7$ from \cite{li_zhang_weighted_estimates} (since, again, we have bounded curvature on $M_\infty \backslash K_\infty$ and $\widetilde{u}_\infty = 0$ on $K_\infty$). This yields a uniform $W^{2,2}_{f_\infty}$-bound on $u_\infty$:
    
        \begin{equation*}
            \left|\left|u_\infty\right|\right|_{W^{2,2}_{f_\infty}\left(M_\infty \backslash B_{g_\infty}\left(p_\infty, 2r_{\mathrm{orb}}\right)\right)} \leq \left|\left|\widetilde{u}_\infty\right|\right|_{W^{2,2}_{f_\infty}\left(M_\infty\right)} \leq C\left(n,R_0,\underline{\mu},\delta\right)\left|\left|u_\infty\right|\right|_{W^{1,2}_{f_\infty}\left(M_\infty\right)}.
        \end{equation*}
    
        A $W^{2,2}_{f_\infty}$-bound for $u_\infty$ on $B_{g_\infty}\left(p_\infty, 2r_{\mathrm{orb}}\right) \backslash K_\infty$ follows from standard local estimates as when we considered the problem on $K_\infty$. This all allows us to deduce that \eqref{weighted_eigenvalue_weak} holds for all $\varphi \in C^\infty_c\left(M_\infty \backslash K_\infty\right)$, hence for all $\varphi \in C^\infty_c\left(M_\infty \backslash \mathcal{Q}\right)$. Proposition \ref{triple_prod_estim}, Lemma \ref{weight_lorentz_bound}, and Lemma \ref{rm_wgt_ineq} then let us use a $2$-capacity argument again (this time see, for instance, Proposition $7$ in \cite{workman_semicontinuity}) to deduce that \eqref{weighted_eigenvalue_weak} holds for all $\varphi \in W^{1,2}_{f_\infty}\left(M_\infty\right)$ after also using a density argument. Therefore, $u_\infty \in \mathcal{E}_{\omega_{\infty,S,\delta}}\left(\lambda_\infty; L_\infty, W^{1,2}_{f_\infty}\left(M_\infty\right)\right)$.
    \subsection{Bubble Region Analysis}\label{bubble_upper}
        We now turn to studying the weighted eigenvalue problem on the ALE bubbles and proceed along the lines of Section \ref{bubble_tree_prelims}. The neck regions will be analyzed in the next section.
    
        We first consider when we have a leaf bubble. Let the associated point-scale sequence be $\left(q^1_i,s^1_i\right)$ with corresponding bubble region $B_{g_i}\left(q^1_i,Rs^1_i\right)$ and $R \gg 1$. Also, define the following for $S < R$:
    
        \begin{align*}
            \widetilde{g}^1_i &:= \left(s^1_i\right)^{-2} g_i,\\
            \widetilde{u}^1_i &:= \left(s^1_i\right)^{\frac{n}{2} - 3}u_i,\\
            \widetilde{\omega}^1_{i,S,\delta}\left(x\right) := \left(s^1_i\right)^2 \omega_{i,S,\delta,1}\left(x\right)
            &= \begin{cases}
                \max\left\{\left(s^1_i\right)^2 \delta^{-2}, d^{-2}_{\widetilde{g}^1_i}\left(x,q^1_i\right)\right\}, &x \in B_{\widetilde{g}^1_i}\left(q^1_i,R\right) \backslash B_{\widetilde{g}^1_i}\left(q^1_i,S\right)\\
                S^{-2}, &x \in B_{\widetilde{g}^1_i}\left(q^1_i,S\right).
            \end{cases} 
        \end{align*}
    
        Note that we only need to consider $\widetilde{\omega}^1_{i,S,\delta,1}$ since we know that we will end up with a leaf bubble, each of which are separable from all other bubbles. We have defined the rescaled quantities $\widetilde{u}^1_i$ and $\widetilde{\omega}^1_{i,S,\delta,1}$ so that the left and right hand sides of \eqref{weighted_eigenvalue_weak} have the same scaling. Importantly, after blowing up around $q^1_i$, we still end up considering an eigenvalue problem on the bubble with a similar weight.
    
        As detailed in Section \ref{bubble_tree_prelims}, $\left(\widetilde{M}^1_i, \widetilde{g}^1_i, q^1_i\right) \rightarrow \left(V^1, h^1, q^1_\infty\right)$ in the smooth pointed Cheeger--Gromov sense and $\left(V^1, h^1\right)$ is a Ricci-flat ALE manifold of order $n$. The smooth convergence tells us $\widetilde{\omega}^1_{i,S,\delta} \rightarrow \widetilde{\omega}^1_{\infty,S}$ in $W^{1,\infty}_{\mathrm{loc}}\left(V^1\right)$, where
    
        \begin{equation*}
            \widetilde{\omega}^1_{\infty,S}\left(x\right) := \begin{cases}
                d^{-2}_{h^1}\left(x,q^1_\infty\right),~~x \in V^1 \backslash B_{h^1}\left(q^1_\infty,S\right)\\
                S^{-2}, ~~ x \in B_{h^1}\left(q^1_\infty,S\right).
            \end{cases}                
        \end{equation*}
    
        Next, we observe that, for $i \gg 1$,
    
        \begin{equation*}
            \int_{B_{\widetilde{g}^1_i}\left(q^1_i,R\right)} \widetilde{\omega}^1_{i,S,\delta} \left|\widetilde{u}^1_i\right|^2~e^{-f_i} dV_{\widetilde{g}^1_i} \leq 2\int_{M_i} \omega_{i,S,\delta}\left|u_i\right|^2~e^{-f_i} dV_{g_i} = 2.
        \end{equation*}
    
        Therefore, since $\widetilde{\omega}^1_{\infty,S} > 0$ on compact sets, we can take $i \gg 1$ so that $s^1_i \ll 1$ and
    
        \begin{equation*}
            \int_{B_{\widetilde{g}^1_i}\left(q^1_i,R\right)} \left|\widetilde{u}^1_i\right|^2 e^{-f_i} dV_{\widetilde{g}^1_i} \leq \frac{2}{\min_{B_{h^1}\left(q^1_\infty,R\right)}\widetilde{\omega}^1_{\infty,S}}\int_{B_{\widetilde{g}^1_i}\left(q^1_i,R\right)} \widetilde{\omega}^1_{i,S,\delta} \left|\widetilde{u}^1_i\right|^2~e^{-f_i} dV_{\widetilde{g}^1_i} \leq C\left(S,R\right) < \infty.
        \end{equation*}
    
        Similarly, since the estimates and $2$-capacity argument from Section \ref{body_upper} tell us that $\left|\left|\nabla u_i\right|\right|_{L^2_{f_i}\left(M_i\right)} \leq C\left(R_0,\delta\right)$ on the entire body region and this persists in the limit, we have, for $i \gg 1$,
    
        \begin{equation*}
            \int_{B_{\widetilde{g}^1_i}\left(q^1_i,R\right)} \left|\nabla \widetilde{u}^1_i\right|^2 e^{-f_i} dV_{\widetilde{g}^1_i} \leq \int_{M_i} \left|\nabla u_i\right|^2 e^{-f_i}dV_{g_i} \leq C\left(S,\delta\right).
        \end{equation*}
    
        We therefore deduce that there exists some $\widetilde{u}^1_\infty \in W^{1,2}_{\mathrm{loc}}\left(V^1\right)$ such that
    
        \begin{equation*}
            \begin{cases}
                \widetilde{u}^1_i &\rightharpoonup \widetilde{u}^1_\infty~~\mathrm{in}~ W^{1,2}_{\mathrm{loc}}\left(V^1\right)\\
                \widetilde{u}^1_i &\rightarrow \widetilde{u}^1_\infty~~\mathrm{in}~L^2_{\mathrm{loc}}\left(V^1\right).
            \end{cases}
        \end{equation*}
    
        Also, we have
    
        \begin{align*}
            \int_{V^1} \left|\nabla \widetilde{u}^1_\infty\right|^2 dV_{h^1} &< \infty,\\
            \int_{V^1} \left|\widetilde{u}^1_\infty\right|^2 \widetilde{\omega}_{\infty,S} dV_{h^1} &\leq 2.
        \end{align*}
    
        This can be written a bit more succinctly as $\widetilde{u}^1_\infty \in W^{1,2}_{\widetilde{\omega}^1_{\infty,S}}\left(V^1\right)$, where $W^{1,2}_{\widetilde{\omega}^1_{\infty,S}}\left(V^1\right)$ is the closure of $C^\infty_c\left(V^1\right)$ with respect to the weighted norm
    
        \begin{equation*}
            \left|\left|\widetilde{u}^1_\infty\right|\right|_{W^{1,2}_{\widetilde{\omega}^1_{\infty,S}}\left(V^1\right)} := \left(\int_{V^1} \sum^1_{k = 0} \left(\widetilde{\omega}^1_{\infty,S}\right)^{1-k}\left|\nabla^k u\right|^2 dV_{h^1} \right)^{\frac{1}{2}}.
        \end{equation*}
    
        By the local smooth convergence on $V^1$, we have, similarly to the body region analysis,
    
        \begin{equation*}
            \int_{V^1} \left<\nabla \widetilde{u}^1_\infty, \nabla \varphi\right> - 2\left<\mathrm{Rm}_{h^1} \ast \widetilde{u}^1_\infty, \varphi\right> - \widetilde{\lambda}^1_\infty \widetilde{\omega}^1_{\infty,S}\left<\widetilde{u}^1_\infty,\varphi\right> dV_{h^1} = 0
        \end{equation*}
    
        for every $\varphi \in C^\infty_c\left(V^1\right)$. Here we have written $\widetilde{\lambda}^1_\infty$ to denote the eigenvalue corresponding to the eigenmode $\widetilde{u}^1_\infty$. Note that, by the analysis in Section \ref{body_upper}, $\left|\widetilde{\lambda}^1_\infty\right| \leq C\left(R_0,S,\delta\right)$.
    
        Next, since $\left|\mathrm{Rm}_{h^1}\right|_{h^1}\left(x\right) \leq Cd^{-n-2}_{h^1}\left(x,q^1_\infty\right)$ along the ALE end of $V^1$, we may pick $S \gg 1$ (and increase $R$ accordingly) so that for every $x \in V^1 \backslash B_{h^1}\left(q^1_\infty,2S\right)$ we have $\left|\mathrm{Rm}_{h^1}\right|\left(x\right) < Cd^{-2}_{h^1}\left(x,q^1_\infty\right) \leq C'\widetilde{\omega}^1_{\infty,S}\left(x\right)$. Next, since $V^1$ is smooth, there is a constant $C'' > 0$, depending only on $S$ and $h^1$, such that $\left|\mathrm{Rm}_{h^1}\right|\left(x\right) \leq C'' \widetilde{\omega}^1_{\infty,S}\left(x\right)$ for all $x \in V^1$. We thus have a constant $C''' > 0$ which can change from line to line and depends only on $h^1, n, \underline{\mu}, R_0, S$ such that
    
        \begin{align}
            &\left|\int_{V^1} \left<\nabla \widetilde{u}^1_\infty, \nabla \varphi\right> - 2\left<\mathrm{Rm}_{h^1} \ast \widetilde{u}^1_\infty, \varphi\right> - \widetilde{\lambda}^1_\infty \widetilde{\omega}^1_{\infty,S}\left<\widetilde{u}^1_\infty,\varphi\right> dV_{h^1}\right| \nonumber\\
            &\leq \left|\left|\nabla \widetilde{u}^1_\infty\right|\right|_{L^2\left(V^1\right)}\left|\left|\nabla \varphi\right|\right|_{L^2\left(V^1\right)} + C''\left(1 + \left|\widetilde{\lambda}^1_\infty\right|\right)\int_{V^1} \widetilde{\omega}^1_{\infty,S}\left|\widetilde{u}^1_\infty\right|\left|\varphi\right| dV_{h^1} \nonumber\\
            &\leq \left|\left|\nabla \widetilde{u}^1_\infty\right|\right|_{L^2\left(V^1\right)}\left|\left|\nabla \varphi\right|\right|_{L^2\left(V^1\right)} + C'''\left(\int_{V^1}\widetilde{\omega}^1_{\infty,S} \left|\widetilde{u}^1_\infty\right|^2 dV_{h^1}\right)^{\frac{1}{2}} \left(\int_{V^1}\widetilde{\omega}^1_{\infty,S}\left|\varphi\right|^2 dV_{h^1}\right)^{\frac{1}{2}} \label{bubble_density_estim}\\
            &\leq \left|\left|\nabla \widetilde{u}^1_\infty\right|\right|_{L^2\left(V^1\right)}\left|\left|\nabla \varphi\right|\right|_{L^2\left(V^1\right)}\nonumber\\
            &+ C'''\left(\left|\left|\widetilde{\omega}^1_{\infty,S}\right|\right|_{L^\infty\left(V^1 \backslash B_{h^1}\left(q^1_\infty,\varepsilon\right)\right)}\left|\left|\varphi\right|\right|^2_{L^2\left(V^1 \backslash B_{h^1}\left(q^1_\infty,\varepsilon\right)\right)} + \left|\left|\widetilde{\omega}^1_{\infty,S}\right|\right|_{L^{\frac{n}{2},\infty}\left(B_{h^1}\left(q^1_\infty,\varepsilon\right)\right)}\left|\left|\varphi\right|\right|^2_{W^{1,2}\left(B_{h^1}\left(q^1_\infty,\varepsilon\right)\right)}\right)^{\frac{1}{2}}.\nonumber
        \end{align}
    
        for $0 < \varepsilon \ll 1$ so that we can apply Proposition \ref{triple_prod_estim}, the proof of which can be adapted to ALE bubbles as described in Remark \ref{bubble_adjustment}. The second and third lines follow from H\"older's inequality, Lemma \ref{rm_wgt_ineq}, and the upper bound on $\left|\widetilde{\lambda}^1_\infty\right|$. The final line is also due to $\left|\left|\left(\widetilde{\omega}^1_{\infty,S}\right)^{\frac{1}{2}}\widetilde{u}^1_\infty\right|\right|_{L^2\left(V^1\right)} \leq 2$ and $\left|\left|\widetilde{\omega}^1_{\infty,S}\right|\right|_{L^{\frac{n}{2},\infty}\left(B_{h^1}\left(q^1_\infty,\varepsilon\right)\right)} < \infty$ for every $\varepsilon > 0$. We may thus apply a density argument to deduce that the weighted eigenvalue problem holds on $V^1$ for all $\varphi \in W^{1,2}\left(V^1\right)$. Note that we could have just used $\widetilde{\omega}^1_{\infty,S} \in L^\infty\left(V^1\right)$ and then applied H\"older's inequality. However, the computation we just did also holds on intermediate bubbles, and in such a case we replace $B_{h^1}\left(q^1_\infty,\varepsilon\right)$ with the union of sufficiently small balls around each orbifold point so that Proposition \ref{triple_prod_estim} applies on each of the small balls.
    
        For subsequent bubbles, consider the point-scale sequence $\left(q^k_i, s^k_i\right)$ and define the following:
    
        \begin{align*}
            \widetilde{g}^k_i &:= \left(s^k_i\right)^{-2} g_i\\
            \widetilde{u}^k_i &:= \left(s^k_i\right)^{\frac{n}{2} - 3}u_i\\
            \widetilde{\omega}^k_{i,S,\delta,\ell}\left(x\right) &:= \left(s^k_i\right)^2 \omega_{i,S,\delta,\ell}\left(x\right). 
        \end{align*}
    
        \sloppy For each $k$ these rescalings correspond to the bubble region $B_{g_i}\left(q^k_i,Rs^k_i\right)$. We then have $\left(\widetilde{M}^k_i, \widetilde{g}^k_i, q^k_i\right) \rightarrow \left(V^k, h^k, q^k_\infty\right)$ in the pointed (orbifold) Cheeger--Gromov sense with (possibly empty) singular set $\mathcal{Q}^k$. If $\left(V^k, h^k\right)$ is a leaf bubble, then $\mathcal{Q}^k = \emptyset$ and the analysis is exactly as before. On the other hand, if $\left(V^k,h^k\right)$ is an intermediate bubble, then we need to slightly modify our earlier argument. The reader might find it easier to first go through the following argument under the assumption that only one intermediate bubble forms. After that it is mainly a matter of introducing more notation to account for a larger number of bubbles forming.
    
        If we consider $B_{\widetilde{g}^k_i}\left(q^k_i, R\right)$, then the rescaled weight is as in the leaf bubble case:
    
        \begin{equation*}
            \widetilde{\omega}^k_{i,S,\delta,k} = \begin{cases}
                \max\left\{\left(s^k_i\right)^2 \delta^{-2}, d^{-2}_{\widetilde{g}^k_i}\left(x,q^k_i\right)\right\}, &x \in B_{\widetilde{g}^k_i}\left(q^k_i,R\right) \backslash B_{\widetilde{g}^k_i}\left(q^k_i,S\right)\\
                S^{-2}, &x \in B_{\widetilde{g}^k_i}\left(q^k_i,S\right).
            \end{cases} 
        \end{equation*}
    
        Since $\left(V^k,h^k\right)$ is now an intermediate bubble we also need to account for its child bubbles. Denote the collection of point-scale sequences associated to the child bubbles of $\left(V^k,h^k\right)$ by $\left\{\left(q^\ell_i, s^\ell_i\right)\right\}^L_{\ell = 1}$. Note that we are temporarily reindexing these bubbles. Then we know, by the discussion in Section \ref{bubble_tree_prelims}, that
    
        \begin{align*}
            0 < M^{-1} &\leq d_{\widetilde{g}^k_i}\left(q^k_i,q^\ell_i\right) \leq M,\\
            \widetilde{s}^\ell_i &:= \frac{s^\ell_i}{s^k_i} \rightarrow 0
        \end{align*}
    
        for some $M > 0$. We do not need to consider the case when $\widetilde{s}^\ell_i$ converges to some non-zero value as the bubbles would then be separable by \eqref{separable}. For $i \gg 1$, we consider a ball of radius $Rs^k_i$ around each $q^\ell_i$. Then we have, after rescaling,
    
        \begin{equation*}
            \widetilde{\omega}^k_{i,S,\delta,\ell} = \begin{cases}
                \max\left\{\left(s^k_i\right)^2 \delta^{-2}, d^{-2}_{\widetilde{g}^k_i}\left(x,q^\ell_i\right)\right\}, &x \in B_{\widetilde{g}^k_i}\left(q^k_i,R\right) \backslash B_{\widetilde{g}^k_i}\left(q^\ell_i,S\widetilde{s}^\ell_i\right)\\
                \left(S\widetilde{s}^\ell_i\right)^{-2}, &x \in B_{\widetilde{g}^k_i}\left(q^\ell_i,S\widetilde{s}^\ell_i\right).
            \end{cases} 
        \end{equation*}
    
        Away from $q^\ell_i$, the rescaled weight function $\widetilde{\omega}^k_{i,S,\delta,\ell}$ will converge in the $W^{1,\infty}_{\mathrm{loc}}$-sense to 
    
        \begin{equation*}
            \widetilde{\omega}^k_{\infty,S,\ell}\left(x\right) := d^{-2}_{h^k}\left(x,q^\ell_\infty\right),
        \end{equation*} 
    
        while $\widetilde{\omega}^k_{i,S,\delta,k}$ will converge in the $W^{1,\infty}_{\mathrm{loc}}$-sense to
    
        \begin{equation*}
            \widetilde{\omega}^k_{\infty,S,k}\left(x\right) := 
            \begin{cases}
                d^{-2}_{\widetilde{g}^k_i}\left(x,q^k_i\right), &x \in V^k \backslash B_{h^k}\left(q^k_\infty,S\right)\\
                S^{-2}, &x \in B_{h^k}\left(q^k_\infty,S\right).
            \end{cases} 
        \end{equation*} 
    
        We then set
    
        \begin{align*}
            \widetilde{\omega}^k_{i,S,\delta}\left(x\right) &:= \max_{\ell = 1,\dots, L,k} \widetilde{\omega}^k_{i,S,\delta,\ell}\left(x\right),\\
            \widetilde{\omega}^k_{\infty,S}\left(x\right) &:= \max_{\ell = 1,\dots, L,k} \widetilde{\omega}^k_{\infty,S,\ell}\left(x\right).
        \end{align*}
    
        The curvature decay along the end of $V^k$ and an adaptation of Proposition \ref{rm_wgt_ineq} shows there exists a constant $\widehat{C} > 0$ depending only on $h^k, S$ such that $\left|\mathrm{Rm}_{h^k}\right|\left(x\right) \leq \widehat{C}\widetilde{\omega}^k_{\infty,S}\left(x\right)$ for every $x \in V^k$. We can now proceed as before to show $\left|\widetilde{\lambda}^k_\infty\right|$ is uniformly bounded and that there exists some $\widetilde{u}^k_\infty \in W^{1,2}_{\mathrm{loc}}\left(V^k \backslash \mathcal{Q}^k\right)$ such that
    
        \begin{equation*}
            \begin{cases}
                \widetilde{u}^k_i &\rightharpoonup \widetilde{u}^k_\infty ~~\mathrm{in}~ W^{1,2}_{\mathrm{loc}}\left(V^k \backslash \mathcal{Q}^k\right)\\
                \widetilde{u}^k_i &\rightarrow \widetilde{u}^k_\infty ~~\mathrm{in}~ L^2_{\mathrm{loc}}\left(V^k \backslash \mathcal{Q}^k\right).
            \end{cases}
        \end{equation*}
    
        That $\widetilde{u}^k_\infty \in W^{1,2}_{\mathrm{loc}}\left(V^k\right)$ follows from the arguments at the end of Section \ref{body_upper}. That is, either using the bounds we derived above and the lower semi-continuity of norms under weak convergence, or a $2$-capacity argument. From there, we can show $\widetilde{u}^k_\infty \in W^{1,2}_{\widetilde{\omega}^k_{\infty,S}}\left(V^k\right)$ along the lines of the analysis in the leaf bubble case. Note that $W^{1,2}_{\widetilde{\omega}^k_{\infty,S}}\left(V^k\right)$ is defined analogously to $W^{1,2}_{\widetilde{\omega}^1_{\infty,S}}\left(V^1\right)$, albeit first by taking the completion of $C^\infty_c\left(V^k \backslash \mathcal{Q}^k\right)$ with respect to the $W^{1,2}_{\widetilde{\omega}^k_{\infty,S}}$-norm and then using a $2$-capacity argument. All of this and $\widetilde{\omega}^k_{i,S,\delta} \rightarrow \widetilde{\omega}^k_{\infty,S}$ in the $W^{1,\infty}_{\mathrm{loc}}\left(V^k \backslash \mathcal{Q}^k\right)$-sense tells us the weighted eigenvalue problem holds on $V^k$ for all $\varphi \in C^\infty_c\left(V^k \backslash \mathcal{Q}^k\right)$. We then use a $2$-capacity and density argument again to conclude the weighted problem holds for all $\varphi \in W^{1,2}\left(V^k\right)$. One then proceeds by induction to complete the blow-up analysis of the eigenvalues/modes on each bubble.
    \subsection{Neck Region Analysis}\label{neck_upper}
        We now show that index and nullity cannot concentrate in the neck regions. Roughly speaking, this should be the case as the neck theorem from \cite{by_bubbling} tells us the neck regions are diffeomorphic to an annulus on a flat cone, hence should be strictly linearly stable.
    
        In the following result and its proof, we denote a neck region by $\mathcal{N}$ to ease notation as the exact bubble scales involved do not really matter. For similar reasons, we will denote the weight function by $\omega_{S,\delta}$. Also, recall that $K := B_g\left(p, r_{\mathrm{orb}}\right)$.
    
        \begin{proposition}\label{spectral_neck}
            For $n \geq 4$, let $\left(M, g, f\right)$ be an $n$-dimensional complete connected gradient Ricci shrinker with $\mu\left(g\right) \geq \underline{\mu} > -\infty$, $R_g + \sup_{M \backslash K} \left|\mathrm{Rm}_g\right| \leq R_0$, and, if $n \geq 5$, satisfying \eqref{energybounds}. Then there exists some $\varepsilon = \varepsilon\left(R_0, \underline{\mu},n,S,\delta\right) > 0$ such that, when $\left|\mathrm{Rm}_g\right| \leq \varepsilon \omega_{S,\delta}$ on $\mathcal{N}$, we have
    
            \begin{equation*}
                0 < \lambda_0 \leq \inf \left\{\mathcal{B}_f\left[u,u\right] : u \in W^{1,2}_{0,f}\left(M\right), \int_{\mathcal{N}} \omega_{S,\delta} \left|u\right|^2~e^{-f} dV_g = 1\right\}.
            \end{equation*} 
    
            Here $W^{1,2}_{0,f}\left(\mathcal{N}\right)$ is the $W^{1,2}_f\left(\mathcal{N}\right)$ closure of $C^1_c\left(\mathcal{N}\right)$ and $\lambda_0 = \lambda_0\left(\varepsilon,R_0,\underline{\mu},n,W\right) > 0$, with $W$ the upper bound for $\left|\left|\omega_{S,\delta}\right|\right|_{L^{\frac{n}{2},\infty}_f\left(\mathcal{N}\right)}$.
        \end{proposition}
    
        \begin{proof}
            By density, it suffices to prove the result for $u \in C^1_c\left(M\right)$ with $\mathrm{supp}\left(u\right) \subseteq \mathcal{N}$ and satisfying the normalization in the statement of the result. Since $u$ has compact support, a slight adjustment to the proof of Proposition \ref{triple_prod_estim}, as mentioned in Remark \ref{triple_prod_remark}, yields
    
            \begin{equation*}
                1 = \int_{\mathcal{N}} \omega_{S,\delta} \left|u\right|^2~e^{-f} dV_g \leq CW\left|\left|\nabla u\right|\right|^2_{L^2_f\left(\mathcal{N}\right)}.
            \end{equation*}
    
            Here $C > 0$ is a constant depening only on $n, \underline{\mu}, R_0$. Note that we also used $\left|\left|\omega_{S,\delta}\right|\right|_{L^{\frac{n}{2},\infty}_f\left(\mathcal{N}\right)} \leq W$ here. We can now estimate as follows:
    
            \begin{align*}
                1 &\leq CW \int_{\mathcal{N}} \left|\nabla u\right|^2 e^{-f} dV_g\\
                &= CW \int_{\mathcal{N}} \left(\left|\nabla u\right|^2 - 2\left<\mathrm{Rm}_g \ast u, u\right>\right) e^{-f} dV_g + 2CW \int_{\mathcal{N}} \left<\mathrm{Rm}_g \ast u, u\right> e^{-f} dV_g\\
                &\leq CW \int_{\mathcal{N}} \left(\left|\nabla u\right|^2 - 2\left<\mathrm{Rm}_g \ast u, u\right>\right) e^{-f} dV_g + 2CW\varepsilon\int_{\mathcal{N}} \omega_{S,\delta} \left|u\right|^2 e^{-f} dV_g\\
                &= CW \int_{\mathcal{N}} \left(\left|\nabla u\right|^2 - 2\left<\mathrm{Rm}_g \ast u, u\right>\right) e^{-f} dV_g + 2CW\varepsilon.
            \end{align*}
    
            Note that the third line is due to $\left|\mathrm{Rm}_g\right|\left(x\right) \leq \varepsilon \omega_{S,\delta}\left(x\right)$ for all $x \in \mathcal{N}$. Next, taking $\varepsilon < \left(4CW\right)^{-1}$ and rearranging yields
    
            \begin{equation*}
                0 < \lambda_0 := \frac{1 - 2CW\varepsilon}{CW} \leq \int_{\mathcal{N}} \left(\left|\nabla u\right|^2 - 2\left<\mathrm{Rm}_g \ast u, u\right>\right) e^{-f} dV_g,
            \end{equation*}
    
            which is what we wanted. 
        \end{proof}
    
        Note that Proposition \ref{spectral_neck} does not depend on the neck region having the same bubble scale in both the inner and outer radii. This is due to the expression we found for the lower bound of $\lambda_0$ as well as the uniform control on the local Lorentz--Sobolev constant and the weight function from, respectively, Lemma \ref{lorentz_sobolev} and Lemma \ref{weight_lorentz_bound}. Finally, note that the assumptions of Proposition \ref{spectral_neck} are satisfied on each neck region after taking $i \gg 1$, hence the neck regions are strictly linearly stable and do not contribute any index or nullity. In particular, recall from the final part of the proof of Lemma \ref{rm_wgt_ineq} that in neck regions $\left|\mathrm{Rm}_g\right|\omega^{-1}_{S,\delta} \rightarrow 0$ as $i \rightarrow \infty$. Hence we can guarantee, for $i \gg 1$, the existence of some $\varepsilon > 0$ so that $\left|\mathrm{Rm}_g\right| \leq \varepsilon \omega_{S,\delta}$ in each neck region.
    
        Before moving on, we introduce some notation to make subsequent results easier to state and prove. Consider an eigenmode $u_i \in \mathcal{E}_{\omega_{i,S,\delta}}\left(\lambda_i; L_{f_i}, W^{1,2}_{f_i}\left(M_i\right)\right)$ with $\lambda_i \leq 0$. Then, as we proved in Section \ref{body_upper} and Section \ref{bubble_upper}, after passing to a subsequence, we have $\lambda_i \rightarrow \lambda_\infty \leq 0$ and
    
        \begin{equation*}
            u_i \rightarrow \left(u_\infty,\left(\widetilde{\mathbf{u}}^1_\infty, \dots, \widetilde{\mathbf{u}}^{\left|\mathcal{Q}\right|}_\infty\right)\right),
        \end{equation*}
    
        with $u_\infty \in W^{1,2}_{f_\infty}\left(M_\infty\right)$ and $\widetilde{\mathbf{u}}^k_\infty$ denoting, for each $k = 1, \dots, \left|\mathcal{Q}\right|$,
    
        \begin{equation*}
            \widetilde{\mathbf{u}}^k_\infty := \left(\widetilde{u}^k_\infty, \left(\widetilde{\mathbf{u}}^{k,1}_\infty, \dots, \widetilde{\mathbf{u}}^{k,\left|\mathcal{Q}^k\right|}_\infty\right)\right).
        \end{equation*}
    
        Here, for each $k$, $\widetilde{u}^k_\infty \in W^{1,2}_{\widetilde{\omega}^k_{\infty,S}}\left(V^k\right)$. The first entries $u_\infty$ and $\widetilde{u}^k_\infty$ are the contributions (respectively) from the body region of the orbifold shrinker $M_\infty$ and the bubble $V^k$, while each $\widetilde{\mathbf{u}}^k_\infty$ and $\widetilde{\mathbf{u}}^{k,\ell}_\infty$ is the contribution from the bubble tree associated to (respectively) $q^k_\infty \in \mathcal{Q}$ and $q^\ell_\infty \in \mathcal{Q}^k$, where $\mathcal{Q}^k$ is the (possibly empty) singular set of $V^k$. This proceeds until one exhausts every bubble tree.
    
        Our next result will show that, due to Proposition \ref{spectral_neck}, we cannot have $\widetilde{\mathbf{u}}^k_\infty = \mathbf{0}$ and $u_\infty = 0$ all at once, where $\mathbf{0}$ means $\widetilde{\mathbf{u}}^k_\infty = \left(0, \left(\mathbf{0},\dots,\mathbf{0}\right)\right)$ and similarly for $\widetilde{\mathbf{u}}^{k,\ell}_\infty$. Now that we have all the technical results and techniques from previous sections in hand, the proof is extremely similar to that of Claim $2$ in \cite{workman_semicontinuity}. We therefore just provide the main points and indicate any needed modifications.
    
        \begin{lemma}\label{main_contradiction_arg}
           For $n \geq 4$, assume $\left(M_i, g_i, f_i\right)$ is a sequence of $n$-dimensional complete connected gradient Ricci shrinkers satisfying $\left(\mathcal{A}\right)$ and let $\left\{a_j\right\}^N_{j=1}$ be a set of real numbers such that $\sum^N_{j=1} a^2_j = 1$. Then, if $\left\{u_{i,j}\right\}^N_{j=1}$ is a set of eigenmodes associated to the non-positive eigenvalues $\left\{\lambda_{i,j}\right\}^N_{j=1}$ which are orthonormal with respect to the weighted inner product $\left<\cdot,\cdot\right>_{\omega_{i,S,\delta},L^2_{f_i}\left(M_i\right)}$ (defined in \eqref{weighted_inner_product}) we have
    
            \begin{equation*}
                v_i := \sum^N_{j=1} a_j u_{j,i} \rightarrow \sum^N_{j=1} a_j \left(u_{j,\infty}, \left(\widetilde{\mathbf{u}}^1_{j,\infty},\dots,\widetilde{\mathbf{u}}^{\left|\mathcal{Q}\right|}_{j,\infty}\right) \right) \neq \left(0,\left(\mathbf{0},\dots,\mathbf{0}\right)\right).
            \end{equation*}
        \end{lemma}
    
        \begin{proof}
            Assume the desired result does not hold and, for simplicity, that only leaf bubbles form. One can run through the following argument when intermediate bubbles form as well. However, the decomposition of $M_i$ used in the definition of the test function $V_i$ we will introduce shortly becomes incredibly tedious to write down.
    
            Consider a cut off function $\chi \in C^\infty\left(\left[0,\infty\right);\left[0,1\right]\right)$ with the following properties:
    
            \begin{itemize}
                \item $\chi\left(t\right) = 1$ for $t \in \left[0,1\right]$,
                \item $\chi\left(t\right) = 0$ for $t \in \left[2,\infty\right)$,
                \item $-C \leq \chi'\left(t\right) \leq 0$,
            \end{itemize}        
    
            where $C > 0$ is a universal constant. Now, for $\rho \ll 1$ and $i \gg 1$, define the following test function, which in particular is non-vanishing only on annular (that is, neck) regions:
    
            \begin{equation}\label{eigenmode_sum_expr}
                V_i\left(x\right) := 
                \begin{cases}
                    0,&x \in M_i \backslash \bigcup_{q^k_i \in \mathcal{Q}} B_{g_i}\left(q^k_i,\rho\right)\\
                    v_i \chi\left(2\rho^{-1}d_{g_i}\left(x,\mathcal{Q}\right)\right),&x \in \bigcup_{q^k_i \in \mathcal{Q}} A^{g_i}_{\frac{\rho}{2},\rho}\left(q^k_i\right)\\
                    v_i,&x \in A^{g_i}_{2Ss^k_i, \frac{\rho}{2}}\left(q^k_i\right)~~\forall q^k_i \in \mathcal{Q} \\
                    v_i \left(1-\chi\left(\left(Ss^k_i\right)^{-1}d_{g_i}\left(x,q^k_i\right)\right)\right),&x \in A^{g_i}_{Ss^k_i, 2Ss^k_i}\left(q^k_i\right)~~\forall q^k_i \in \mathcal{Q}\\
                    0,&x \in B_{g_i}\left(q^k_i,Ss^k_i\right)~~\forall q^k_i \in \mathcal{Q}.
                \end{cases}
            \end{equation}
    
            The triangle inequality then tells us we need to estimate
    
            \begin{align}
                \left|\mathcal{B}_{f_i}\left(v_i,v_i\right) - \mathcal{B}_{f_i}\left(V_i,V_i\right)\right| &\leq \int_{M_i} \left|\left|\nabla v_i\right|^2 - \left|\nabla V_i\right|^2\right| e^{-f_i}dV_{g_i} \label{bilinear_diff}\\
                &+ \int_{M_i} \left(\omega_{i,S,\delta} + \left|\mathrm{Rm}_{g_i}\right|\right)\left|\left|v_i\right|^2 - \left|V_i\right|^2\right| e^{-f_i} dV_{g_i}\nonumber.
            \end{align}
    
            We aim to show, for $\rho \ll 1$ and $S, i \gg 1$, that the right hand side \eqref{bilinear_diff} can be made arbitrarily small. Then, since $V_i$ satisfies the assumptions of Proposition \ref{spectral_neck}, $\mathcal{B}_{f_i}\left(V_i,V_i\right) \geq \lambda_0 > 0$, we must have $\mathcal{B}_{f_i}\left(v_i,v_i\right) > 0$ which contradicts $\lambda_{i,k} \leq 0$. To adjust $V_i$ to the general case one needs to excise child bubbles from each intermediate bubble and define the function appropriately in terms of the bubble scales on these removed regions.
    
            To accomplish the goal stated above, one estimates as in the proof of Claim $2$ in \cite{workman_semicontinuity}. For the first term on the right hand side of \eqref{bilinear_diff}, this involves only using the properties of $\chi$, the definition of $V_i$, and Young's inequality. Estimating the second term uses the same strategy, as well as Lemma \ref{rm_wgt_ineq} and that $\omega_{i,S,\delta}, \mathrm{Rm}_{g_i}$ are bounded outside the bubble regions. This latter point means the standard H\"older's inequality can be used on the body region rather than Proposition \ref{triple_prod_estim}. In the end, one arrives at, for $i \gg 1$ and $\rho \ll 1$,
    
            \begin{align*}
                &\left|\mathcal{B}_{f_i}\left(v_i,v_i\right) - \mathcal{B}_{f_i}\left(V_i,V_i\right)\right|\\
                &\leq C\left|\left|\nabla v_i\right|\right|^2_{L^2_{f_i}\left(M_i \backslash \bigcup_{q^k_i \in \mathcal{Q}} B_{g_i}\left(q^k_i,\rho\right)\right)}\\
                &+ C\sum_{q^k_i \in \mathcal{Q}} \left(\left|\left|\nabla v_i\right|\right|^2_{L^2_{f_i}\left(A^{g_i}_{\frac{\rho}{2}, \rho}\left(q^k_i\right)\right)} + \rho^{-2} \left|\left|v_i\right|\right|^2_{L^2_{f_i}\left(A^{g_i}_{\frac{\rho}{2},\rho}\left(q^k_i\right)\right)} + \left|\left|\widetilde{v}^k_i\right|\right|^2_{W^{1,2}\left(B_{\widetilde{g}^k_i}\left(q^k_i, 2S\right)\right)}\right).
            \end{align*}
    
            Here $\widetilde{v}^k_i$ denotes the rescaling of $v_i$ by the bubble scale $s^k_i$ as in Section \ref{bubble_upper}. Note that \eqref{weighted_eigenvalue_weak} is satisfied on any $B_{g_i}\left(x,\frac{\rho}{2}\right) \subset \subset M_i \backslash \bigcup_{q^k_i \in \mathcal{Q}} B_{g_i}\left(q^k_i, 2Ss^k_i\right)$. Therefore, as at the end of Section \ref{body_upper}, we may appeal to Theorem $4.9$ in \cite{giaquinta} to deduce that
    
            \begin{equation}\label{w22_regularity}
                \left|\left|v_i\right|\right|_{W^{2,2}_{f_i}\left(B_{g_i}\left(x, \frac{\rho}{4}\right)\right)} \leq C\left(R_0,\delta,S,\rho\right)\left(\left|\left|v_i\right|\right|_{L^2_{f_i}\left(B_{g_i}\left(x,\frac{\rho}{2}\right)\right)} + \rho^{-2}\left|\left|\sum^N_{j=1}\lambda_{i,j} a_j u_{i,j}\right|\right|_{L^2_{f_i}\left(B_{g_i}\left(x,\frac{\rho}{2}\right)\right)}\right).
            \end{equation} 
    
            Then Proposition \ref{weighted_sobolev_compact} and our assumption $v_i \rightarrow \left(0,\left(\mathbf{0},\dots,\mathbf{0}\right)\right)$ tells us 
    
            \begin{equation*}
                \begin{cases}
                    v_i \rightharpoonup 0,&\mathrm{in}~W^{2,2}_{f_\infty}\left(B_{g_\infty}\left(x,\frac{\rho}{4}\right)\right)\\
                    v_i \rightarrow 0,&\mathrm{in}~W^{1,2}_{f_\infty}\left(B_{g_\infty}\left(x,\frac{\rho}{4}\right)\right).
                \end{cases}
            \end{equation*}
    
            Note that the interior estimates yielding this hold on every bounded subset of the body region. However, $M_i$ is non-compact in general so we need to ensure we have good control globally. This follows from, as at the end of Section \ref{body_upper}, deducing uniform $W^{1,2}_{f_i}$ bounds on $M_i \backslash K_i$ in terms of the $W^{1,2}_{f_i}$-norm of $v_i$ by using a cut-off argument and appealing to Proposition $3.7$ in \cite{li_zhang_weighted_estimates}. Proposition \ref{weighted_sobolev_compact} then yields strong $W^{1,2}_{f_i}$-convergence on $M_i \backslash K_i$. One can argue similarly on the bubble regions after using the rescaling from Section \ref{bubble_upper}. Finally, cover each annular region in the definition of $V_i$ by a finite collection of balls and apply \eqref{w22_regularity} to the covering. In all, for any $\varepsilon > 0$ we can take $\rho \ll 1$ and $i,S \gg 1$ so that
    
            \begin{equation*}
                \left|\mathcal{B}_{f_i}\left(v_i,v_i\right) - \mathcal{B}_{f_i}\left(V_i,V_i\right)\right| < \varepsilon.
            \end{equation*}
    
            Then each annulus in the definition of $V_i$ satisfies the assumptions of Proposition \ref{spectral_neck}. Taking $\varepsilon < \frac{\lambda_0}{2}$, with $\lambda_0 > 0$ from Proposition \ref{spectral_neck}, yields a contradiction by the reasoning from earlier.
        \end{proof}
    \subsection{Equivalence of the Weighted and Unweighted Eigenspaces}\label{equivalence}
      Now that we have completed our bubbling analysis of the eigenvalues/modes, it remains to show the dimensions of our weighted eigenspaces and unweighted eigenspaces are the same. On an (orbifold) shrinker we have the following.
    
      \begin{proposition}\label{body_equivalence_result}
        For $n \geq 4$, assume $\left(M_i, g_i, f_i\right)$ is a sequence of $n$-dimensional complete connected gradient Ricci shrinkers satisfying $\left(\mathcal{A}\right)$. Then
    
        \begin{align*}
          \mathrm{span}\left\{\bigcup_{\substack{j \\ \lambda_{i,j} \leq 0}} \mathcal{E}\left(\lambda_{i,j}; L_{f_i}, W^{1,2}_{f_i}\left(M_i\right)\right)\right\} &= \bigoplus_{\substack{j \\ \lambda_{i,j} \leq 0}} \mathcal{E}\left(\lambda_{i,j}; L_{f_i}, W^{1,2}_{f_i}\left(M_i\right)\right),\\
          \mathrm{span}\left\{\bigcup_{\substack{\ell \\ \zeta_{i,\ell} \leq 0}} \mathcal{E}_{\omega_{i,S,\delta}}\left(\zeta_{i,\ell}; L_{f_i}, W^{1,2}_{f_i}\left(M_i\right)\right)\right\} &= \bigoplus_{\substack{\ell \\ \zeta_{i,\ell} \leq 0}} \mathcal{E}_{\omega_{i,S,\delta}}\left(\zeta_{i,\ell}; L_{f_i}, W^{1,2}_{f_i}\left(M_i\right)\right),
        \end{align*}
    
        and
    
        \begin{equation*}
          \sum_{\substack{\ell \\ \zeta_{i,\ell} \leq 0}} \mathrm{dim}\left(\mathcal{E}_{\omega_{i,S,\delta}}\left(\zeta_{i,\ell}; L_{f_i}, W^{1,2}_{f_i}\left(M_i\right)\right)\right) = \sum_{\substack{j \\ \lambda_{i,j} \leq 0}} \mathrm{dim}\left(\mathcal{E}\left(\lambda_{i,j}; L_{f_i}, W^{1,2}_{f_i}\left(M_i\right)\right)\right).
        \end{equation*} 
    
        Here $\omega_{i,S,\delta}: M \rightarrow \mathbb{R}$ is the weight function described in Section \ref{weighted_problem} while the unions, summations, and direct sums are taken over all $j$ such that $\lambda_{i,j} \leq 0$ is an eigenvalue of $L_{f_i}$ and all $\ell$ such that $\zeta_{i,\ell} \leq 0$ is an eigenvalue of the corresponding weighted problem. The analogous conclusions hold on the limiting orbifold shrinker $M_\infty$.
      \end{proposition}
    
      \begin{proof}
        We have all the necessary technical tools in hand to proceed almost exactly as in the proof of Proposition $8$ in \cite{workman_semicontinuity}, so we largely omit the details. Showing
    
        \begin{equation}\label{body_equiv_bnd}
          \sum_{\substack{j \\ \zeta_{i,j} \leq 0}} \mathrm{dim}\left(\mathcal{E}_{\omega_{i,S,\delta}}\left(\zeta_{i,j}; L_{f_i}, W^{1,2}_{f_i}\left(M_i\right)\right)\right) \leq \sum_{\substack{j \\ \lambda_{i,j} \leq 0}} \mathrm{dim}\left(\mathcal{E}\left(\lambda_{i,j}; L_{f_i}, W^{1,2}_{f_i}\left(M_i\right)\right)\right)
        \end{equation} 
    
        can be done as in \cite{workman_semicontinuity}. To prove the reverse inequality, we just need to adapt Workman's argument to construct possible eigenmode candidates. This can be done as in \cite{workman_semicontinuity} when $\mathrm{Rm}_g \in L^\infty\left(M\right)$, as is the case when we deal with the smooth shrinkers $M_i$. On the other hand, when working with the orbifold shrinker $M_\infty$, one can pass to local orbifold covers and proceed as in the smooth case since, as outlined in Remark \ref{orbifold_remark}, $M_\infty$ is an orbifold of bounded curvature. Another argument involving adaptations of weighted Sobolev spaces from \cite{dai_wang1,dai_wang2} is also viable and we outline this in Section \ref{lower_proof}. 
    
        A third argument, which we use to make the proof more self-contained, involves adapting Workman's argument as follows when orbifold points are present. We omit subscripts for simplicity and proceed using a contradiction argument. Set
    
        \begin{align*}
          I &:= \mathrm{dim}\left(\bigoplus_{\lambda \leq 0} \mathcal{E}\left(\lambda; L_f, W^{1,2}_f\left(M\right)\right)\right)\\
          \mathcal{U} &:= \mathrm{span}\left\{u_1, \dots, u_I\right\} \subset W^{1,2}_f\left(M\right)
        \end{align*}
    
        with $\left\{u_j\right\}^I_{j=1}$ orthonormal with respect to $\left<\cdot,\cdot\right>_{\omega, L^2_f\left(M\right)}$. Then we may write $W^{1,2}_f\left(M\right) = \mathcal{U} \oplus \mathcal{U}^{\perp_\omega}$. Assume for a contradiction that there is a linear subspace $\widetilde{\mathcal{U}} \subset W^{1,2}_f\left(M\right)$ with $\mathrm{dim}\left(\widetilde{\mathcal{U}}\right) = I + 1$ and $\left.\mathcal{B}_f\right|_{\widetilde{\mathcal{U}}} \leq 0$. Then the projection map $P_{\mathcal{U}} : \widetilde{\mathcal{U}} \rightarrow \mathcal{U}$ has a non-trivial kernel and there is some $v \in \widetilde{\mathcal{U}} \cap \mathcal{U}^{\perp_\omega}$ with $\left<v,v\right>_{\omega, L^2_f\left(M\right)}$. We now consider
    
        \begin{equation*}
          \widetilde{\lambda} := \inf\left\{\mathcal{B}_f\left[u,u\right] : u \in \mathcal{U}^{\perp_\omega}, \left<u,u\right>_{\omega,L^2_f\left(M\right)} = 1\right\} \leq 0.
        \end{equation*}
    
        Take a sequence $\widetilde{u}_k \in \mathcal{U}^{\perp_\omega}$ such that $\left<\widetilde{u}_k,\widetilde{u}_k\right>_{\omega,L^2_f\left(M\right)} = 1$ and
    
        \begin{equation*}
          \widetilde{\lambda} = \lim \limits_{k \rightarrow \infty} \mathcal{B}_f\left[\widetilde{u}_k,\widetilde{u}_k\right].
        \end{equation*}
    
        We want to now guarantee
    
        \begin{equation*}
          \begin{cases}
                  \widetilde{u}_k \rightharpoonup \widetilde{u}~~\mathrm{in}~W^{1,2}_f\left(M\right)\\
                  \widetilde{u}_k \rightarrow \widetilde{u}~~\mathrm{in}~L^2_f\left(M\right),
              \end{cases}
        \end{equation*}
    
        for some limit $\widetilde{u} \in W^{1,2}_f\left(M\right)$ which will act as a candidate eigenmode. To do this we derive a uniform lower bound on $\widetilde{\lambda}$ and an upper bound on the $W^{1,2}_f$-norm of $\widetilde{u}_k$ as in Section \ref{body_upper}, this time using $\left|\mathrm{Rm}_g\right|\left(x\right) \leq C'\omega\left(x\right)$ for all $x \in M$ as guaranteed by the second part of Lemma \ref{rm_wgt_ineq}. From there, we can proceed exactly as in the proof of Proposition $8$ in \cite{workman_semicontinuity} to conclude the desired result.
      \end{proof}
      
      Now for the analogous result on our ALE bubbles. For this, note that Theorem \ref{bubble_index_energy_bound} tells us the unweighted index and nullity are finite for every bubble. Also, recall that the stability operator for ALE bubbles is the usual Einstein operator $Lu = \Delta u + 2\mathrm{Rm}_h \ast u$, since the blow-up procedure yields $f \equiv C$ for some constant $C \in \mathbb{R}$ on each bubble.
    
      \begin{proposition}\label{bubble_index_equivalence_result}
        For $n \geq 4$, let $\left(V,h\right)$ be an $n$-dimensional ALE bubble as in Definition \ref{ale_bubble} with (possibly empty) singular set $\mathcal{Q}$ and let $\omega: V \rightarrow \mathbb{R}$ be a weight function on $V$ as described in Section \ref{bubble_upper}. Then
    
        \begin{align*}
          \mathrm{span}\left\{\bigcup_{\lambda \leq 0} \mathcal{E}\left(\lambda; L, W^{1,2}\left(V\right)\right)\right\} &= \bigoplus_{\lambda \leq 0} \mathcal{E}\left(\lambda; L, W^{1,2}\left(V\right)\right),\\
          \mathrm{span}\left\{\bigcup_{\lambda \leq 0} \mathcal{E}_\omega\left(\lambda; L, W^{1,2}_\omega\left(V\right),W^{1,2}\left(V\right)\right)\right\} &= \bigoplus_{\lambda \leq 0} \mathcal{E}_\omega\left(\lambda; L, W^{1,2}_\omega\left(V\right),W^{1,2}\left(V\right)\right),
        \end{align*}
    
        and
        
        \begin{equation*}
          \sum_{\lambda < 0} \mathrm{dim}\left(\mathcal{E}_\omega\left(\lambda; L, W^{1,2}_\omega\left(V\right), W^{1,2}\left(V\right)\right)\right) 
          = \sum_{\lambda < 0} \mathrm{dim}\left(\mathcal{E}\left(\lambda; L, W^{1,2}\left(V\right)\right)\right)
        \end{equation*}
      \end{proposition}
    
      \begin{proof}
        We need the bi-linear form and weighted $L^2$-inner product to be well-defined on the bubbles. First, recall that $\left|\mathrm{Rm}_h\right|_h \leq C \omega$ on $V \backslash K_h$, for some compact set $K_h \subset V$. Also, for all $\varepsilon > 0$, $\omega, \mathrm{Rm}_h \in L^{\frac{n}{2},\infty}\left(\bigcup_{q \in \mathcal{Q}} B_h\left(q,\varepsilon\right)\right)$ and are each bounded on $V \backslash \bigcup_{q \in \mathcal{Q}} B_h\left(q,\varepsilon\right)$. Therefore, we can estimate as in the derivation of \eqref{bubble_density_estim} to guarantee that $\mathcal{B}\left[u,\varphi\right]$ and $\left<u,\varphi\right>_{\omega, L^2\left(V\right)}$ are well-defined, as long as $u \in W^{1,2}_\omega\left(V\right)$ and $\varphi \in W^{1,2}\left(V\right)$. Note that each $u \in W^{1,2}_{\omega}\left(V\right)$ is only in $L^2_{\mathrm{loc}}$. However, Lemma \ref{finite_index_capture} tells us we can restrict the analysis to a compact set which captures all of the index and then proceed as in the proof of Proposition $9$ of \cite{workman_semicontinuity}.
      \end{proof}
    
      We now want to show the equivalence of the weighted and unweighted nullities on the bubbles. We cannot merely proceed along the lines of \cite{workman_semicontinuity} as in the proof of Proposition \ref{bubble_index_equivalence_result}. In particular, domain monotonicity of eigenvalues means an analogue of Lemma \ref{finite_index_capture} cannot hold for the nullity. However, using Proposition \ref{ale_fredholm}, we can prove the following:
    
      \begin{proposition}\label{bubble_nullity_equivalence_result}
        For $n \geq 4$, let $\left(V,h\right)$ be an $n$-dimensional ALE bubble and consider the weight function $\omega$ on the bubble $V$ as outlined in Section \ref{bubble_upper}. Then
    
        \begin{equation*}
          \mathcal{E}_{\omega}\left(0; L, W^{1,2}_\omega\left(V\right), W^{1,2}\left(V\right)\right) = \mathcal{E}\left(0; L, W^{1,2}\left(V\right)\right).
        \end{equation*}
      \end{proposition}
    
      \begin{proof}
        Using the definition of $W^{1,2}_\omega\left(V\right)$, we see that $u \in W^{1,2}_{\frac{n}{2} - 1}\left(V\right)$ (see Definition \ref{bubble_wgtd_sobolev}). Next, since the bubble is ALE of order $n$, one can show that $u \in W^{2,2}_{\frac{n}{2}-1}\left(V\right)$ and satisfies the estimate
    
        \begin{equation*}
          \left|\left|u\right|\right|_{W^{2,2}_{\frac{n}{2}-1}\left(V\right)} \leq C\left(\left|\left|Lu\right|\right|_{L^2_{\frac{n}{2} + 1}\left(V\right)} + \left|\left|u\right|\right|_{L^2_{\frac{n}{2} - 1}\left(V\right)} \right)
        \end{equation*}
    
        for a constant $C = C\left(n, h, S\right) > 0$. This follows similarly to the proof of Proposition $1.6$ in \cite{bartnik} through applying local $L^p$ estimates and a scaling technique. This argument can be iterated to show any $u \in \ker_{L^2_{\frac{n}{2}-1}}\left(L\right)$ is actually in $W^{k,2}_{\frac{n}{2}-1}$ for all $k \in \mathbb{N}$.
    
        Now consider our stability operator $L$ as a map between weighted Sobolev spaces:
    
        \begin{equation*}
          L: W^{2,2}_{\frac{n}{2}-1}\left(V\right) \rightarrow L^2_{\frac{n}{2} + 1}\left(V\right).
        \end{equation*}
    
        Since $n \geq 4$ we can use Proposition \ref{ale_fredholm} to deduce that the kernel of $L$ between these weighted spaces coincides with the standard $L^2$-kernel of $L$. In other words, the weighted and unweighted nullities we are considering are the same.
      \end{proof}
    
      \begin{remark}
        Let $3 \leq n \leq 6$. Then the proof of Proposition \ref{bubble_nullity_equivalence_result} carries over to Workman's setting in \cite{workman_semicontinuity}, where each bubble $V$ is a minimal hypersurface in $\mathbb{R}^{n+1}$ with finite total curvature. Lemma $4$ in \cite{tysk} tells us that the finite total curvature condition implies each end of $V$ is a hyperplane of multiplicity one. Then Proposition $3$ in \cite{schoen_reg_infty} tells us that, along each end, there are constants $C,C'> 0$ such that $\left|A\right| \leq Cr^{-n}$ and each end is a graph of some function $u$ over some plane such that $\left|u\right| \leq C'r^{-n+2}$. This means the bubbles and stability operator in Workman's setting satisfy the conditions to apply the results from \cite{bartnik} used in the proof of Proposition \ref{bubble_nullity_equivalence_result}. Furthermore, as each end of $V$ is a hyperplane, the exceptional values for the Laplacian are the same as in $\mathbb{R}^n$. This all tells us the conclusion of Proposition \ref{bubble_nullity_equivalence_result} holds for the bubbles arising in \cite{workman_semicontinuity}.
      \end{remark}
    \subsection{Proof of the Upper Semi-Continuity Estimate}
      We are now able to prove the upper semi-continuity estimate in Theorem \ref{semi_continuity_thm}, whose statement we we recall for the reader's convenience.
    
      \begin{theorem}
        For $n \geq 4$, assume $\left(M_i, g_i, f_i\right)$ is a sequence of $n$-dimensional complete connected gradient Ricci shrinkers satisfying $\left(\mathcal{A}\right)$. Then
    
        \begin{equation*}
          \limsup_{i \rightarrow \infty} \left(\mathrm{Ind}_{f_i}\left(M_i\right) + \mathrm{Null}_{f_i}\left(M_i\right)\right) \leq \mathrm{Ind}_{f_\infty}\left(M_\infty\right) + \mathrm{Null}_{f_\infty}\left(M_\infty\right) + \sum_{q \in \mathcal{Q}} \sum^{N_q}_{k=1} \mathrm{Ind}\left(V^k\right) + \mathrm{Null}\left(V^k\right).
        \end{equation*}
    
        Here the index of the bubbles is defined as $\mathrm{Ind}\left(V^k\right) := \lim \limits_{R \rightarrow \infty} \mathrm{Ind}\left(B_{h^k}\left(q^k,R\right)\right)$ and the nullity is defined analogously.
      \end{theorem}
    
      \begin{proof}
        If $\limsup \limits_{i \rightarrow \infty} \left(\mathrm{Ind}_{f_i}\left(M_i\right) + \mathrm{Null}_{f_i}\left(M_i\right)\right) = 0$ then we are done. Thus, without loss of generality we may assume that there is some $N \geq 1$ such that
    
        \begin{equation*}
          N \leq \limsup \limits_{i \rightarrow \infty} \left(\mathrm{Ind}_{f_i}\left(M_i\right) + \mathrm{Null}_{f_i}\left(M_i\right)\right) = \limsup \limits_{i \rightarrow \infty} \left(\sum_{\lambda_i \leq 0} \mathrm{dim}\left(\mathcal{E}_{\omega_{i,S,\delta}}\left(\lambda_i; L_{f_i}, W^{1,2}_{f_i}\left(M_i\right)\right)\right)\right),
        \end{equation*}
    
        where the equality is due to Proposition \ref{body_equivalence_result}. Then, possibly after passing to a subsequence, we get a linear subspace $V_i$ such that
    
        \begin{equation*}
          V_i := \mathrm{span}\left\{u_{k,i}\right\}^N_{k=1} \subset W^{1,2}_{f_i}\left(M_i\right)
        \end{equation*}
    
        where, for each $k$, $u_{k,i}$ is the eigenmode associated to the eigenvalue $\lambda_{k,i} \leq 0$. Also, we have
    
        \begin{equation*}
          u_{k,i} \in \mathcal{E}_{\omega_{i,S,\delta}}\left(\lambda_{k,i}; L_{f_i}, W^{1,2}_{f_i}\left(M_i\right)\right)
        \end{equation*}
    
        and $\left\{u_{k,i}\right\}^N_{k=1}$ is orthonormal with respect to $\left<\cdot,\cdot\right>_{\omega_{i,S,\delta},L^2_{f_i}}$. Next, define $E_\infty$ as
    
        \begin{align*}
          E_\infty := \bigoplus_{\lambda_\infty \leq 0} &\mathcal{E}_{\omega_{\infty,S,\delta}}\left(\lambda_\infty; L_{f_\infty}, W^{1,2}_{f_\infty}\left(M_\infty\right)\right)\\
          &\oplus \bigoplus_{q \in \mathcal{Q}} \bigoplus^{N_q}_{k = 1} \bigoplus_{\widetilde{\lambda}^k_{\ell,\infty} \leq 0}\mathcal{E}_{\widetilde{\omega}^k_{\infty,S}}\left(\widetilde{\lambda}^k_{\ell,\infty}; L_{h^k}, W^{1,2}_{\widetilde{\omega}^k_{\infty, S}}\left(V^k\right), W^{1,2}\left(V^k\right)\right).
        \end{align*}
    
        Strictly speaking we should define $E_\infty$ in terms of the spans of the various eigenspaces, but the first parts of Proposition \ref{body_equivalence_result} and Proposition \ref{bubble_index_equivalence_result} show that this is equivalent to how we have expressed $E_\infty$. Define a linear map $\Pi_i$ by
    
        \begin{align*}
          \Pi_i: V_i &\rightarrow E_\infty\\
          u_{k,i} &\mapsto \left(u_{k,\infty},\left(\widetilde{\mathbf{u}}^1_{k,\infty}, \dots, \widetilde{\mathbf{u}}^{\left|\mathcal{Q}\right|}_{k,\infty}\right)\right)
        \end{align*}
    
        and
    
        \begin{equation*}
          v_i := \sum^N_{k=1} a_k u_{k,i}
        \end{equation*}
    
        with $\sum^N_{k=1} a^2_k = 1$. If
    
        \begin{equation*}
          \Pi_i\left(v_i\right) = \left(0,\left(\mathbf{0},\dots,\mathbf{0}\right)\right) 
        \end{equation*}
    
        then we would have
    
        \begin{align*}
          v_i &\rightarrow \left(0,\left(\mathbf{0},\dots,\mathbf{0}\right)\right),\\
          \int_{M_i} &\left|v_i\right|^2 \omega_{i,S,\delta} ~e^{-f_i} dV_{g_i} = 1.
        \end{align*}
    
        This contradicts Lemma \ref{main_contradiction_arg}, hence $\Pi_i$ is injective and, for $i \gg 1$,
    
        \begin{equation*}
          N = \mathrm{dim}\left(V_i\right) \leq \mathrm{dim}\left(E_\infty\right).
        \end{equation*}
    
        We then conclude by using Proposition \ref{body_equivalence_result}, Proposition \ref{bubble_index_equivalence_result}, and Proposition \ref{bubble_nullity_equivalence_result} to swap from the weighted index and nullity to the unweighted versions.
      \end{proof}
\section{Proof of the Lower Semi-Continuity Results}\label{lower_proof}
  The final section of this paper is devoted to proving the lower semi-continuity part of Theorem \ref{semi_continuity_thm}, as well as Theorem \ref{cone_lower_bound}. We now recall the first of these results for the reader's convenience. The proof of is an adaptation of an argument in \cite{karpukhin_stern}.

  \begin{theorem}\label{lower_semi_cont_proof}
    For $n \geq 4$, assume $\left(M_i, g_i, f_i\right)$ is a sequence of $n$-dimensional complete connected gradient Ricci shrinkers satisfying $\left(\mathcal{A}\right)$. Then

    \begin{equation*}
      \mathrm{Ind}_{f_\infty}\left(M_\infty\right) + \sum_{q \in \mathcal{Q}} \sum^{N_q}_{k=1} \mathrm{Ind}\left(V^k\right) \leq \liminf_{i \rightarrow \infty} \mathrm{Ind}_{f_i}\left(M_i\right).
    \end{equation*}

    Here the index of the bubbles is defined as $\mathrm{Ind}\left(V^k\right) := \lim \limits_{R \rightarrow \infty} \mathrm{Ind}\left(B_{h^k}\left(q^k,R\right)\right)$.
  \end{theorem}

  \begin{proof}
    Let $I := \mathrm{Ind}_{f_\infty}\left(M_\infty\right)$. We then know that there is an $I$-dimensional subspace $\mathcal{U}_\infty \subset L^2_{f_\infty}\left(M_\infty\right)$ and a $\beta > 0$ such that for all $0 \not\equiv u_\infty \in \mathcal{U}_\infty$ we have

    \begin{equation*}
      \mathcal{B}_{f_\infty}\left[u_\infty,u_\infty\right] < -\beta \left|\left|u_\infty\right|\right|^2_{L^2_{f_\infty}\left(M_\infty\right)}. 
    \end{equation*}

    The idea is to now use cutoff functions to perturb $\mathcal{U}_\infty$ to a new $I$-dimensional subspace $\mathcal{U}'_\infty$ such that any $u'_\infty \in \mathcal{U}'_\infty$ vanishes on the bubble regions $\bigcup_{q_\infty \in \mathcal{Q}} B_{g_\infty}\left(q_\infty,\rho\right)$ and

    \begin{equation}\label{perturbed_bound}
      \max_{0 \neq u'_\infty \in \mathcal{U}'_\infty} \frac{\mathcal{B}_{f_\infty}\left[u'_\infty,u'_\infty\right]}{\left|\left|u'_\infty\right|\right|^2_{L^2_{f_\infty}\left(M_\infty\right)}} < -\frac{\beta}{2}. 
    \end{equation}

    We then repeat the argument on the smooth shrinkers $M_i$ so that we get an $I$-dimensional space $\mathcal{U}'_i$ consisting of elements supported away from the bubble regions. On the other hand, $\mathrm{dim}\left(\mathcal{U}'_i\right) \leq \mathrm{Ind}_{f_i}\left(M_i\right)$ by construction and domain monotonicity of the eigenvalues. One then uses the smooth convergence of the shrinker and eigenmodes away from the bubble regions and iterates the proof (after rescaling) on each bubble region to get the desired conclusion.

    We first work on the orbifold shrinker $\left(M_\infty,g_\infty,f_\infty\right)$ and consider the cut-off function $\chi_\rho$ with the following properties:

    \begin{itemize}
      \item $\chi_\rho = 1$ on $M_\infty \backslash \bigcup_{q_\infty \in \mathcal{Q}} B_{g_\infty}\left(q_\infty, \rho\right)$,\\
      \item $\chi_\rho = 0$ on $\bigcup_{q_\infty \in \mathcal{Q}} B_{g_\infty}\left(q_\infty, \frac{\rho}{2}\right)$,
      \item $\left|\nabla \chi_\rho\left(x\right)\right| \leq Cd^{-1}_{g_\infty}\left(x,\mathcal{Q}\right)$ for some universal constant $C > 0$.
    \end{itemize}

    Set $u'_\infty := \chi_\rho u_\infty$. One can show that, as $\rho \rightarrow 0$,

    \begin{equation*}
      \max \left\{\mathcal{B}_{f_\infty}\left[u'_\infty,u'_\infty\right] : u_\infty \in \mathcal{U}_\infty, \left|\left|u_\infty\right|\right|_{L^2_{f_\infty}\left(M_\infty\right)} = 1\right\} \rightarrow \max_{0 \not\equiv u_\infty \in \mathcal{U}_\infty} \frac{\mathcal{B}_{f_\infty}\left[u_\infty,u_\infty\right]}{\left|\left|u_\infty\right|\right|^2_{L^2_{f_\infty}\left(M_\infty\right)}} < -\beta.
    \end{equation*}

    Then, since linear independence is an open condition, we see that for $\rho \ll 1$ the space

    \begin{equation*}
      \mathcal{U}'_\infty := \left\{u'_\infty : u_\infty \in \mathcal{U}_\infty\right\}
    \end{equation*}

    is an $I$-dimensional subspace of $L^2_{f_\infty}\left(M_\infty\right)$, the elements of which are supported away from $\bigcup_{q_\infty \in \mathcal{Q}} B_{g_\infty}\left(q_\infty,\rho\right)$ and satisfy \eqref{perturbed_bound}. We have suppressed anydependence $\mathcal{U}'_\infty$ has on $\rho$ to keep the notation simpler.

    Now we perturb the eigenmodes along the sequence and make sure they converge to the $u'_\infty$ above. We know that $\left(M_i,g_i,f_i\right) \rightarrow \left(M_\infty,g_\infty,f_\infty\right)$ in the $C^\infty_{\mathrm{loc}}$-sense away from the bubble regions. Also, away from the bubble regions, for each eigemode $u_i$ we have $u_i \rightarrow u_\infty$ strongly in the $W^{k,2}_f$-sense for all $k \in \mathbb{N}$. This is due to Proposition \ref{weighted_sobolev_compact}, as well as Proposition $3.7$ from \cite{li_zhang_weighted_estimates} and the cut-off argument at the end of Section \ref{neck_upper}.

    Next, possibly after shrinking $\rho$, we can guarantee the existence of cut-off functions $\chi_{\rho,i}$ with the same properties as $\chi_\rho$ and $\chi_{\rho,i} \rightarrow \chi_\rho$ in the $C^1$-sense. Now, for $i \gg 1$ and $\rho \ll 1$, let $u'_i := \chi_{\rho,i} u_i$ and

    \begin{equation*}
      \mathcal{U}'_i := \left\{u'_i : u'_\infty \in \mathcal{U}'_\infty\right\}.
    \end{equation*}

    Note that, by the discussion about the convergence of eigenmodes from earlier and $\mathrm{Rm}_{g_\infty}$ being bounded away from the bubble regions regions, we have $\mathcal{B}_{f_i}\left[u'_i,u'_i\right] \rightarrow \mathcal{B}_{f_i}\left[u'_\infty,u'_\infty\right]$ as $i \rightarrow \infty$.

    \sloppy Since this convergence is uniform on the compact set $\left\{u'_\infty \in \mathcal{U}'_\infty : \left|\left|u'_\infty\right|\right|_{L^2_{f_\infty}\left(M_\infty\right)} = 1\right\}$ we see that, for $i \gg 1$ and $\rho \ll 1$, $\mathcal{U}'_i$ is an $I$-dimensional subspace of $L^2_{f_i}\left(M_i\right)$, the elements of which are supported away from $\bigcup_{q_i \in \mathcal{Q}} B_{g_i}\left(q_i,\rho\right)$ and satisfy \eqref{perturbed_bound}.

    The full theorem now follows from successively applying the above argument to each node in the bubble tree after rescaling. To accomplish this one needs to first work with the weighted eigenvalue problem \eqref{weighted_eigenvalue_weak} and then appeal to Proposition \ref{body_equivalence_result} and Proposition \ref{bubble_index_equivalence_result}.
  \end{proof}

  Now we prove Theorem \ref{cone_lower_bound}, which we also recall for the reader's convenience:

  \begin{theorem}
    \sloppy For $n \geq 4$, let $\left(M,g,f\right)$ be an $n$-dimensional complete connected gradient Ricci shrinker with finitely many ends which is asymptotically conical to the cone $\left(\mathcal{C}\left(\Sigma\right), g_{\mathcal{C}} := dr^2 + r^2 g_\Sigma\right)$ with vertex $p_{\mathcal{C}}$, where the cone link $\left(\Sigma, g_\Sigma\right)$ is an $\left(n-1\right)$-dimensional closed manifold. Then there is a continuous function $f_{\mathcal{C}}: \mathcal{C}\left(\Sigma\right) \rightarrow \mathbb{R}$ so that if $L_{f_{\mathcal{C}}}: L^2_{f_{\mathcal{C}}}\left(\mathcal{C}\left(\Sigma\right)\right) \rightarrow L^2_{f_{\mathcal{C}}}\left(\mathcal{C}\left(\Sigma\right)\right)$ is upper semi-bounded we have

    \begin{equation*}
      \mathrm{Ind}_{f_\mathcal{C}}\left(\mathcal{C}\left(\Sigma\right)\right) \leq \mathrm{Ind}_f\left(M\right).    
    \end{equation*}
  \end{theorem}

  \begin{proof}
    Recall that every gradient Ricci shrinker induces an ancient Ricci flow $g\left(t\right) = -t \varphi^\ast_t g\left(-1\right)$, where $t \in \left(-\infty,0\right)$ and $\varphi_t$ is the family of diffeomorphisms generated by $\frac{\nabla f}{-t}$. Also, the associated time dependent potential function is $f\left(t\right) = \varphi^\ast_t f$.

    Now, if $M$ is asymptotically conial we have

    \begin{equation*}
      \left(M, g\left(t\right), -tf\left(t\right), p_t\right) \rightarrow \left(\mathcal{C}\left(\Sigma\right), g_{\mathcal{C}}, f_{\mathcal{C}}, p_{\mathcal{C}}\right)
    \end{equation*} 

    as $t \nearrow 0$, where $p_t$ is a minimum of $f\left(t\right)$. This convergence is globally in the pointed Gromov--Hausdorff sense and the smooth pointed Cheeger--Gromov sense away from the cone vertex $p_{\mathcal{C}}$ (Proposition $2.1(3)$ in \cite{kw_conical_rigidity}), which tells us $\mathcal{C}\left(\Sigma\right)$ is a regular cone. That is, it is smooth away from the vertex. We also have $f_{\mathcal{C}} = \lim_{t \nearrow 0} t f\left(t\right)$ where the convergence is in the $C^\infty_{\mathrm{loc}}$-sense away from the cone vertex and uniformly in the $C^0$-sense on all of $\mathcal{C}\left(\Sigma\right)$. Moreover, outside a compact set, $f_{\mathcal{C}} = \frac{r^2}{4}$, has no critical points, and 

    \begin{align}
      \Delta f_{\mathcal{C}} &= \frac{n}{2} \label{cone_traced},\\
      \left|\nabla f_{\mathcal{C}}\right|^2 &= f_{\mathcal{C}}\label{cone_auxiliary}.      
    \end{align}

    We also note that the proof of \eqref{cone_traced} and \eqref{cone_auxiliary} involve considering the corresponding versions of \eqref{grs_traced} and \eqref{grs_auxiliary} on $\left(M_i, g_i, f_i\right)$ and considering the limit as $t \nearrow 0$.
    For a detailed proof of these results we refer the reader to the proof of Theorem $2.1$ in \cite{kw_conical_rigidity} as well as Section $4.3$ of \cite{siepmann_thesis}. This latter reference concerns expanding gradient Ricci solitons, but the methods carry over to the shrinker setting and are similar to those in \cite{kw_conical_rigidity}.

    By considering a sequence of times $t_i \nearrow 0$ we can produce a sequence of Ricci shrinkers converging to $\mathcal{C}\left(\Sigma\right)$ as above:

    \begin{equation*}
      \left(M_i, g_i:= -t_i \varphi^\ast_{t_i}g\left(-1\right), f_i := -t_i f\left(t_i\right), p_i\right) \rightarrow \left(\mathcal{C}\left(\Sigma\right), g_{\mathcal{C}}, f_{\mathcal{C}}, p_{\mathcal{C}}\right)
    \end{equation*}

    where $p_i$ is a minimum of $f_i$ as usual.

    Now, as in the proof of Corollary $4.2$ in \cite{dai_wang1}, since $L_{f_{\mathcal{C}}}: L^2_{f_{\mathcal{C}}}\left(\mathcal{C}\left(\Sigma\right)\right) \rightarrow L^2_{f_{\mathcal{C}}}\left(\mathcal{C}\left(\Sigma\right)\right)$ is upper semi-bounded by assumption it admits a self-adjoint extension with domain lying in $W^{1,2}_{f_{\mathcal{C}}, \omega_{\infty,\delta}}\left(\mathcal{C}\left(\Sigma\right)\right)$, which is defined analogously to the weighted space $W^{1,2}_{\widetilde{\omega}^k_{\infty,S}}\left(V^k\right)$ from Section \ref{bubble_upper}. On the other hand, one can show the associated weighted norm is equivalent to the usual Sobolev norm on $W^{1,2}_{f_{\mathcal{C}}}\left(\mathcal{C}\left(\Sigma\right)\right)$ by using a Hardy inequality as in \cite{dai_wang2} and that $e^{-f_{\mathcal{C}}} dV_{g_{\mathcal{C}}}$ and $dV_{g_{\mathcal{C}}}$ are equivalent on compact subsets, in this case $B_{g_{\mathcal{C}}}\left(p_{\mathcal{C}},\delta\right)$. Moreover, $W^{1,2}_{f_{\mathcal{C}}}\left(\mathcal{C}\left(\Sigma\right)\right)$ compactly embeds into $L^2_{f_{\mathcal{C}}}\left(\mathcal{C}\left(\Sigma\right)\right)$. This follows from a slight variation of the proof of Proposition \ref{weighted_sobolev_compact} involving the use of \eqref{cone_traced} and \eqref{cone_auxiliary}.

    Putting everything together, one can show the analogue of Theorem \ref{spectral_theorem} holds for $L_{f_{\mathcal{C}}}$ on the cone $\mathcal{C}\left(\Sigma\right)$. Importantly, this tells us that $\mathrm{Ind}_{f_{\mathcal{C}}}\left(\mathcal{C}\left(\Sigma\right)\right) < \infty$ and we have a variational characterization of the eigenvalues. Therefore, we may appeal to (the proof of) Proposition \ref{body_equivalence_result} to ensure the weighted and unweighted indexes are equal in the following. Note this also uses that Lemma \ref{rm_wgt_ineq} holds on $\mathcal{C}\left(\Sigma\right)$ as mentioned in Remark \ref{bubble_adjustment}.

    We can now proceed along the lines of part of the proof of Theorem $1.3$ in \cite{barbosa_sharp_wei}. Note first that for all $t < 0$, $I := \mathrm{Ind}_f\left(M\right) = \mathrm{Ind}_{-tf\left(t\right)}\left(M\right)$ because of the self-similarity of the Ricci flow induced by the shrinker (though the eigenvalues themselves could change in magnitude). The only time the index can change is at the singular time $t = 0$. Suppose for a contradiction that $\mathrm{Ind}_{f_\mathcal{C}}\left(\mathcal{C}\left(\Sigma\right)\right) \geq I + 1$. The proof of Lemma \ref{finite_index_capture} in \cite{tysk} tells us there is some $R \gg 1$ such that $\mathrm{Ind}_{f_\mathcal{C}}\left(B_{g_{\mathcal{C}}}\left(p_\mathcal{C}, \rho\right)\right) \geq I+1$ for all $\rho \geq R$. By (the proof of) Proposition \ref{body_equivalence_result} this also holds for the weighted index as mentioned earlier. Thus there are symmetric $2$-tensors $u^k$ that are compactly supported in $B_{g_{\mathcal{C}}}\left(p_\mathcal{C}, R\right)$ such that 

    \begin{align}
      \left<u^k,u^\ell\right>_{\omega_{\infty,\delta}, L^2_{f_\mathcal{C}}\left(\mathcal{C}\left(\Sigma\right)\right)} &= \delta_{k\ell} \label{cone_orthnorm}\\
      \mathcal{B}_{f_{\mathcal{C}}}\left[u^k,u^k\right] &= \lambda^k < 0 \nonumber
    \end{align}

    for all $k,l = 1,\dots,I+1$. Extend each $u^k$ to all of $\mathcal{C}\left(\Sigma\right)$ in an arbitrary $C^1$ fashion and denote by $\left\{u^k_i\right\}^{I+1}_{k=1}$ the eigenmodes on $M_i$ which converge, at least on $B_{g_{\mathcal{C}}}\left(p_\mathcal{C}, R\right)$, to $\left\{u^k\right\}^{I+1}_{k=1}$. This can be made more precise by repeating the analysis in Section \ref{body_upper} and noting we can also find compactly supported eigenmodes on each $M_i$ by the proof of Lemma \ref{finite_index_capture} in \cite{tysk}.

    Next, without loss of generality, we can assume the ordering $\lambda^1_i \leq \lambda^2_i \leq \dots < \lambda^{I+1}_i < 0$ for $i \gg 1$. We then take $\chi: \mathbb{R}_+ \rightarrow \mathbb{R}_+$ to be a cut-off function such that, for $R' > 0$,

    \begin{itemize}
      \item $0 \leq \chi \leq 1$,
      \item $\chi\left(\rho\right) = 1$ for $\rho \in \left[0,R'\right]$,
                \item $\chi\left(\rho\right) = 0$ for $\rho \in \left[2R',\infty\right)$,
      \item $\left|\chi'\right| \leq \frac{C}{R'}$ for a universal constant $C > 0$,
    \end{itemize}

    Then, for $R' \gg R$, $\mathcal{B}_{f_i}\left[\chi u^k_i, \chi u^k_i\right] \leq \frac{\lambda^{I+1}_i}{2} < 0$ for each $k$ and all $i \gg 1$. Since the support of each $\chi u^k_i$ lies in $B_{g_i}\left(p_i, 2R'\right)$ and $\mathrm{Ind}_{f_i}\left(M_i\right) = I$, the set $\left\{\chi u^k_i\right\}^{I+1}_{k=1}$ must be linearly dependent for all $i \gg 1$. We can then assume, without loss of generality, that

    \begin{equation*}
      \chi u^{I+1}_i = a^1_i \chi u^1_i + \dots + a^{I+1}_i \chi u^{I+1}_i
    \end{equation*}

    with each $\left|a^k_i\right| \leq 1$. This all yields

    \begin{equation*}
      \lim \limits_{i \rightarrow \infty} \left<\chi u^k_i, \chi u^\ell_i\right>_{\omega_{i,S,\delta},L^2_{f_i}\left(M_i\right)} = \left<\chi u^k, \chi u^\ell\right>_{\omega_{\infty,\delta},L^2_{f_{\mathcal{C}}}\left(\mathcal{C}\left(\Sigma\right)\right)} = \delta_{k\ell}
    \end{equation*}

    for each $k,\ell = 1, \dots, I+1$. Using this and the linear dependence of $\left\{\chi u^k_i\right\}^{I+1}_{k=1}$ we have, for all $k$,

    \begin{equation*}
      0 = \lim \limits_{i \rightarrow \infty} \left<\chi u^k_i, \chi u^\ell_i\right>_{\omega_{i,S,\delta},L^2_{f_i}\left(M_i\right)} = \lim \limits_{i \rightarrow \infty} a^k_i.
    \end{equation*}

    However, this means that

    \begin{equation*}
      \lim \limits_{i \rightarrow \infty} \int_{M_i} \omega_{i,S,\delta} \left|\chi u^{I+1}_i\right|^2~e^{-f_i} dV_{g_i} = \lim \limits_{i \rightarrow \infty} \sum^I_{k=1} \left|a^k_i\right|^2 \int_{M_i} \omega_{i,S,\delta} \left|\chi u^k_i\right|^2 e^{-f_i}~dV_{g_i} = 0,
    \end{equation*}

    which contradicts \eqref{cone_orthnorm}. Thus, after appealing to (the proof of) Proposition \ref{body_equivalence_result}, $\mathrm{Ind}_{f_{\mathcal{C}}}\left(\mathcal{C}\left(\Sigma\right)\right) \leq \mathrm{Ind}_f\left(M\right) = I$ as desired.
  \end{proof}

  \begin{remark}
    It would of course be desirable to not have to assume $L_{f_{\mathcal{C}}}: L^2_{f_{\mathcal{C}}}\left(\mathcal{C}\left(\Sigma\right)\right) \rightarrow L^2_{f_{\mathcal{C}}}\left(\mathcal{C}\left(\Sigma\right)\right)$ is upper semi-bounded, or more vitally that it admits a self-adjoint extension. However, proving this directly seems to be a delicate matter. In particular, there is a Schr\"odinger type operator with potential $\psi$ such that $0 \leq \psi\left(r\right) + \frac{\left(n-1\right)\left(n-3\right)}{4r^2} \leq \frac{c}{r^2}$ for $c < \frac{3}{4}$ that \textit{does not} admit a self-adjoint extension on $C^\infty_c\left(\mathbb{R}^n \backslash \left\{0\right\}\right)$. For details, we refer the reader to Theorem $X.11$ in \cite{reed_simon2}. Since $L_{f_\mathcal{C}}$ could have a similar form on a cone, it is not entirely clear to the author how to bypass this issue or adapt other arguments in this paper, at least in an easy/straightforward manner. Other than directly finding examples of shrinkers for which $L_{f_{\mathcal{C}}}$ is upper semi-bounded, a couple ways forward could be to find certain conditions on the cone link so that the techniques of Dai--Wang in \cite{dai_wang1} can be adapted to treat the drift Lichnerowicz Laplacian on cones, or  which guarantee ``tangential stability'' (see \cite{kroncke_vertman}).
  \end{remark}

\textbf{Statements and Declarations:} No conflicts of interests to declare.

\end{document}